\documentclass[11pt,a4paper,leqno]{amsproc}

\usepackage{anysize}
\marginsize{3.5cm}{3.5cm}{2.5cm}{2.5cm}

\usepackage[]{hyperref}
\hypersetup{
    colorlinks=true,       
    linkcolor=red,          
    citecolor=blue,        
    filecolor=magenta,      
    urlcolor=cyan           
}

\usepackage{amsmath}
\usepackage{graphicx}
\usepackage{amsfonts,amssymb}
\usepackage{enumerate}
\usepackage{mathrsfs}
\usepackage{amsthm}

\theoremstyle{plain}
\newtheorem{theo}{Theorem}[section]
\newtheorem{prop}[theo]{Proposition}
\newtheorem{lemm}[theo]{Lemma}
\newtheorem{coro}[theo]{Corollary}

\newtheorem{defi}[theo]{Definition}
\theoremstyle{definition}
\newtheorem{rema}[theo]{Remark}
\newtheorem{nota}[theo]{Notation}
\newtheorem{exam}[theo]{Example}

\DeclareMathOperator{\supp}{supp}

\DeclareSymbolFont{pletters}{OT1}{cmr}{m}{sl}
\DeclareMathSymbol{s}{\mathalpha}{pletters}{`s}

\def\defn{\mathrel{:=}}

\def\eps{\varepsilon}

\def\la{\left\lvert}
\def\lA{\left\lVert}
\def\le{\leq}
\def\les{\lesssim}

\def\mez{\frac{1}{2}}
\def\mezl{\frac{3}{2}}
\def\partialx{\nabla}

\def\ra{\right\rvert}
\def\rA{\right\rVert}

\def\tdm{\frac{3}{2}}

\def\xC{\mathbf{C}}
\def\xN{\mathbf{N}}
\def\xR{\mathbf{R}}
\def\xS{\mathbf{S}}

\def\xZ{\mathbf{Z}}

\def\ld{\lambda}
\def\ep{\varepsilon}

\def\cF{\mathcal{F}}

\newcommand{\bq}{\begin{equation}}
\newcommand{\eq}{\end{equation}}
\newcommand{\bqa}{\begin{eqnarray*}}
\newcommand{\eqa}{\end{eqnarray*}}
\newcommand{\hk}{\hspace*{.24in}}

\numberwithin{equation}{section}

\title{A pseudo-local property of gravity water waves system}
\author{
Quang-Huy Nguyen
\address{Quang Huy Nguyen. Laboratoire de Math\'ematiques d'Orsay, UMR 8628 du CNRS, Universit\'e Paris-Sud, 91405 Orsay Cedex, France}
}
\date{\empty}
\begin{document}
\begin{abstract}
By proving a weighted contraction estimate in uniformly local Sobolev spaces for the flow of gravity water waves, we show that this  nonlocal system is in fact pseudo-local in the following sense: locally in time, the dynamic far away from a given bounded region has a small effect on that region (again, in a sense that we will make precise in the article). Our estimate on the flow also implies a new spatial decay property of the waves. To prove this result, we establish a  paradifferential calculus theory in uniformly local Sobolev spaces with weights. 
\end{abstract} 
\thanks{The author was supported in part by Agence Nationale de la Recherche
  project  ANA\'E ANR-13-BS01-0010-03.}
\maketitle
\section{Introduction}
\subsection{The problem} 
We consider an incompressible, irrotational, inviscid fluid moving in a domain $\Omega$ underneath a free surface described by $\eta$ and above a bottom described by a  given function $\eta_*$, which is assumed to be bounded and continuous. Namely,
\bq\label{defi:domain}
\Omega = \{(t,x,y) \in[0,T] \times \xR^d \times \xR: \eta_* (x)< y < \eta(t,x)\} .
\eq
We also denote by $\Sigma$ the free surface and by $\Gamma$ the bottom,
 \begin{align*}
 \Sigma &= \{(t,x,y) \in [0,T] \times \xR^d \times \xR:  y= \eta(t,x))\},\\
 \Gamma &=  \{(x,y) \in \xR^d \times \xR: y= \eta_*(x)\}.
 \end{align*}
The velocity filed $v$ admits a potential $\phi: \Omega \to \xR$ 
such that $v = \nabla_{x,y} \phi$ and $\Delta_{x,y} \phi = 0$ in $\Omega$. We introduce the trace of the potential on the surface
$$
\psi(t,x)= \phi(t,x,\eta(t,x))
$$
and the Dirichlet-Neumann operator 
\bq\label{def:G}
\begin{aligned}
G(\eta) \psi &= \sqrt{1 + \vert \nabla_x \eta \vert ^2}
\Big( \frac{\partial \phi}{\partial n} \Big \arrowvert_{\Sigma}\Big)\\
&= (\partial_y \phi)(t,x,\eta(t,x)) - \nabla_x \eta(t,x) \cdot(\nabla_x \phi)(t,x,\eta(t,x)).
\end{aligned}
\eq
Then (see \cite{LannesLivre}) the gravity water waves system in the Zakharov/Craig--Sulem formulation reads as follows
\begin{equation}\label{ww}
\left\{
\begin{aligned}
\partial_t \eta &= G(\eta) \psi,\\
\partial_t \psi &= - \mez \vert \nabla_x \psi \vert^2 + \mez \frac{(\nabla_x \eta \cdot \nabla_x \psi + G(\eta)\psi)^2}{1+ \vert \nabla_x \eta \vert^2} - g \eta
\end{aligned}
\right.
\end{equation}
where $g$ is the acceleration of gravity.\\
Following \cite{ABZ2} we shall consider the vertical and horizontal components of the velocity on the free surface as unknowns which can be expressed in terms of $\eta$ and $\psi$:
\begin{equation}\label{BV}
B = (v_y)\arrowvert_\Sigma = \frac{ \nabla_x \eta \cdot \nabla_x \psi + G(\eta)\psi} {1+ \vert \nabla_x \eta \vert^2},\quad
V= (v_x)\arrowvert_\Sigma  =\nabla_x \psi - B \nabla_x \eta.
 \end{equation}
 Recall also that the Taylor coefficient $a = -\frac{\partial P}{\partial y}\big\arrowvert_\Sigma$
 can be defined in terms of $\eta,\psi,B,V$ only (see \S 4.2 in \cite{ABZ1} and \S 4.3.1 in \cite{LannesLivre}).\\
\hk  The (local) well-posedness theory for gravity water waves (under the formulation \eqref{ww} or the others) in Sobolev spaces $H^s(\xR^d)$  has been studied by many authors, for example Yosihara \cite{Yosihara1982}, Wu \cite{WuInvent, WuJAMS}, Lannes \cite{LannesJAMS}; we refer to the recent book of Lannes \cite{LannesLivre} for a comprehensive survey of the subject. In these works, the waves were assumed to be of infinite extend (and vanish at infinity), that is, there is no restriction on the horizontal direction. However, in reality water waves always propagate in some bounded container (a lake, an ocean, etc) and hence there will be contacts at the "vertical boundary" of the container. A natural question then arises: (Q) can we justify the $\xR^d$-approximation? More precisely, if \eqref{ww} is a good model then it has to satisfy in certain sense the following property: the dynamic at "infinity" has a small effect on bounded regions.  Since \eqref{ww} appears to be nonlocal (due to the presence of the Dirichlet-Neumann operator) it is not clear that the above replacement at "infinity" is harmless. We should mention that in the special case of a canal or a rectangle basin where the walls are right vertical, the local theory was considered by Alazard-Burq-Zuily \cite{ABZ2}, Kinsey-Wu \cite{KinseyWu}, Wu \cite{Wu2015}.  Our goal in the present paper is to give the following answer to question (Q). Considering a bounded reference  domain, we shall prove that in some sense, far away from this reference domain, the dynamic there has a small effect on the reference domain, and the farther it is the smaller the effect is. In other words, this proves that the gravity water waves system enjoys the "pseudo-local property" (the terminology "pseudo" will be clear in our explanation below).
\subsection{Main results}
We recall first the definition of uniformly local Sobolev spaces (or Kato's spaces) introduced by Kato in \cite{Kato}. 
\begin{defi}\label{Kato}
Let $\chi \in C ^\infty(\xR^d)$ with
$\supp \chi \subset [-1,1]^d, \chi = 1$ near $ [-\frac{1}{4},\frac{1}{4}]^d$   such that 
\begin{equation}\label{kiq}
\sum_{q \in \xZ^d}\chi_q(x) = 1,\quad \forall   x\in \xR^d,~\chi_q(x) = \chi(x-q).
\end{equation} 
For $s\in \xR$ define $H^s_{ul}(\xR^d)$ the space of distributions $u\in H^s_{loc}(\xR^d)$  such that 
$$\Vert u\Vert_{H^s_{ul}(\xR^d)}:= \sup_{q\in \xZ^d}\Vert \chi_q u\Vert_{H^s (\xR^d)} < +\infty.$$
\end{defi}
This definition is independent of the choice of  $\chi\in C_0^\infty(\xR^d)$ satisfying~\eqref{kiq} (see Lemma 7.1 in \cite{ABZ2}).
Let us now define the classes of weights that we will consider.
\begin{defi}\label{classW}
1. We define the class $\mathcal{W}$ of acceptable weights  to be the class of all functions $w:\xR^d\to (0, \infty)$ satisfying the following conditions:
\begin{enumerate}
\item[(i)]  $r_1:=\frac{\nabla w^{-1}}{w^{-1}}$ and $r'_1:=\frac{\nabla w}{w}$ belong to $C^\infty_b(\xR^d)$, where $w^{-1}(x)=1/w(x)$,\\
\item [(ii)] for any $C_1>0$, there exists $C_2>0$ such that for any $x_0\in \xR^d$, there hold 
\[
w(x)\le C_2w(x_0)\quad\text{and}\quad w(x)^{-1}\le C_2w(x_0)^{-1}\quad\forall x\in \xR^d,~|x-x_0|\le C_1,
\]
\end{enumerate}
2. If $w\in \mathcal{W}$ and there exist $\varrho\ge 0,~C>0$ such that for any $x,~y\in \xR^d$ we have $w(x)w^{-1}(y)\le C\langle x-y\rangle^\varrho$ then we say that $w\in \mathcal{W}_{po}(\varrho)$.\\
3.  If $w\in \mathcal{W}$ and there exist $\varrho\ge 0,~C>0$ such that for any $x,~y\in \xR^d$ we have $w(x)w^{-1}(y)\le C\exp(\rho\langle x-y\rangle)$  then we say that $w\in \mathcal{W}_{ex}(\varrho)$.
\end{defi}
\begin{exam}
For any $t, s\in \xR,~C>1$, the functions $ \langle x\rangle ^s,~\ln(C+|x|^2)$ belong to $\mathcal{W}_{po}$, and the functions $e^{t\langle x\rangle},~e^{t\langle x\rangle}\langle x\rangle ^s$ belong to  $\mathcal{W}_{ex}$ but not to any class $\mathcal{W}_{po}$ if $t\ne 0$. See Remark \ref{rema:Wpo} for further remarks.
\end{exam}
\begin{nota}\label{widetildechi}
1.  $\widetilde{\chi}$ denotes a function in $C^{\infty}_0(\xR^d)$ such that $\widetilde{\chi}=1$ on the support of $\chi$ in definition \ref{Kato}. For every $k\in \xZ^d$, we also Define for $x\in \xR^d$, $\widetilde \chi_k(x):=\widetilde \chi(x-k)$.
2. We set for all $\sigma\in \xR$,
\begin{align*}
\mathcal{H}_{ul}^{\sigma}&=H_{ul}^{\sigma+\mez}(\xR^d)\times H_{ul}^{\sigma+\mez}(\xR^d)\times H_{ul}^\sigma(\xR^d)\times H_{ul}^\sigma(\xR^d),\\
\mathcal{W^\sigma}&=W^{\sigma+\mez, \infty}(\xR^d)\times W^{\sigma+\mez, \infty}(\xR^d)\times W^{\sigma, \infty}(\xR^d)\times W^{\sigma, \infty}(\xR^d).
\end{align*}
Denote also by $U=(\eta, \psi, B, V)$ the unknown of system \eqref{ww} and by $U^0=(\eta^0, \psi^0, B^0, V^0)$ its initial value.
\end{nota}
The Cauchy theory proved in \cite{ABZ2} reads as follows
\begin{theo}\label{theo:ABZ}
Let $s>1+\frac{d}{2}$ and $U^0\in \mathcal{H}_{ul}^s$ with
\bq\label{condition:data}
\inf_{x\in \xR^d} (\eta^0(x)  -\eta_*(x))\geq 2 h>0, \quad \inf_{x\in \xR^d}a(0, x) \geq 2c>0.
\eq
Then there exists $T>0$ such that the Cauchy problem for \eqref{ww} with  datum $U^0$  has a unique solution  
\[
U\in L^\infty([0,T], \mathcal{H}_{ul}^s)\cap C^0([0,T], \mathcal{H}_{ul}^{r}),\quad \forall r<s
\]
and
\bq\label{condition:solution}
\inf_{t\in[0, T]}\inf_{x\in \xR^d}[\eta(t,x) - \eta_*(x)] \geq  h,\quad \inf_{t\in [0, T]}\inf_{x\in \xR^d}a(t,x) \geq  c.
\eq
Moreover, for given  $h,~c>0$ the existence time $T$ can be chosen uniformly for data belonging to a bounded set of $\mathcal{H}_{ul}^s$.
\end{theo}
Conditions \eqref{condition:data} mean that initially, the free surface is away from the bottom and the Taylor coefficient is positively away from $0$. Then the conclusion \eqref{condition:solution} asserts that these properties are propagated by the waves, locally in time. We shall always consider in the sequel solutions of \eqref{ww} obeying these properties, which for the sake of simplicity is denoted by 
\begin{multline}
\mathcal{P}_{s, T}(h, c):=\{ U=(\eta, \psi, B, V)\in L^\infty([0, T], \mathcal{H}^s_{ul})~\text{solution to}~\eqref{ww} ,~\text{satisfying}~\eqref{condition:solution}\\
\text{and}~U\arrowvert_{t=0}~\text{satisfies}~\eqref{condition:data}\}.
\end{multline}
Our main result concerning the solution map of the gravity water waves is stated in the following theorem.
\begin{theo}\label{main}
Let $s>1+{\frac{d}{2}}$, $T>0$ and two positive constants $h, c$.  Then for every $w\in \mathcal{W}_{po}(\varrho),~\varrho\ge 0$  there exists a function $\mathcal{K}:\xR^+\times\xR^+\to \xR^+$ nondecreasing in each argument, such that 
\bq\label{eq.lipschitz}
\Vert w(U_1-U_2)\Vert _{C([0, T], \mathcal{H}_{ul}^{s-1})} \leq \mathcal{K} (M_1, M_2) \Vert w(U_1-U_2)\arrowvert_{t=0}\Vert_{\mathcal{H}_{ul}^{s-1}}
\eq
for all $U_1,~U_2\in \mathcal{P}_{s,T}(h, c)$, provided that the right-hand side is finite, where
\[
M_j:=\lA U_j\rA_{L^\infty([0, T], \mathcal{H}^s_{ul})}<+\infty,\quad j=1, 2.
\]
\end{theo} 
As a consequence, we have 
\begin{coro}\label{coro:main}
 Let $s>1+{\frac{d}{2}}$; $h,c>0$ and $\mathcal{A}$ be a bounded set in $\mathcal{H}_{ul}^s$. Denote by $T$ the uniform existence time of solutions to \eqref{ww} in $\mathcal{P}_{s,T}(h, c)$ with data in $\mathcal{A}$. Then there exists $0<T_1\le T$ such that the following property holds:\\
 for every $w\in \mathcal{W}_{po}(\varrho),~\varrho\ge 0$ one can find a constant $C>0$ such that
\bq\label{eq.lipschitz1}
\Vert w (U_1-U_2)\Vert _{C([0, T_1], \mathcal{H}_{ul}^{s-1})} \leq C \Vert w (U_1-U_2)\arrowvert_{t=0}\Vert_{\mathcal{H}_{ul}^{s-1}}, 
\eq
for all $U_j\in \mathcal{P}_{s,T}(h, c)$ with $U_j\arrowvert_{t=0}\in \mathcal{A}$ and  provided that the right-hand side is finite.
\end{coro}
In Corollary \ref{coro:main} if we take $U_2\arrowvert_{t=0}=0$  and use the Sobolev embeddings (see Proposition 2.2, \cite{ABZ2}) 
\[
H^r_{ul} (\xR^d)\hookrightarrow W^{r-\frac{d}{2}, \infty}(\xR^d),\quad r>\frac{d}{2},~r-\frac{d}{2}\notin \xN,
\]
we derive 
\begin{coro}\label{coro2}
 Let $s>1+{\frac{d}{2}}$ and $h, c>0$. Then  for any    bounded set $\mathcal{A}$ in $\mathcal{H}_{ul}^s$, there exists a time $T>0$ such that:\\
for every $w\in \mathcal{W}_{po}(\varrho),~\varrho\ge 0$  one can find a constant $C>0$ such that
\bq\label{decay:H}
\Vert wU\Vert _{C([0, T], \mathcal{H}_{ul}^{s-1})} \leq C \Vert wU\arrowvert_{t=0}\Vert_{\mathcal{H}_{ul}^{s-1}}
\eq
for all $U\in \mathcal{P}_{s,T}(h, c)$ with $U\arrowvert_{t=0}\in \mathcal{A}$ and provided that the right-hand side is finite. Moreover, if $s\ge r>1+\frac{d}{2}$ and $r-\frac{d}{2}\notin \xN$ it follows that
\bq\label{decay}
\Vert w U\Vert _{C([0, T], \mathcal{W}^{r-1-\frac{d}{2}})} \leq C \Vert wU\arrowvert_{t=0}\Vert_{\mathcal{H}_{ul}^{s-1}}.
\eq
\end{coro}
\begin{rema}
 If $w\in C^\infty_b(\xR^d)$ then the right-hand sides of \eqref{eq.lipschitz},~\eqref{eq.lipschitz1},~\eqref{decay:H},~\eqref{decay} are automatically finite.
\end{rema}
\begin{rema}
Persistence properties in weighted spaces have been studied extensively for asymptotic models of water waves in different regimes: (generalized) Korteweg-de Vries equation equations, Schr\"odinger equations, Benjamin-Ono equation, Camassa-Holm equation,...We refer to the works of  Brandolese \cite{Brand}, Bona-Saut \cite{BoSa}, Fonseca-Linares-Ponce \cite{FoLiPo}, Nahas-Ponce \cite{NaPo}, Ni-Zhou \cite{NiZh}.
\end{rema}
\begin{rema}
It is natural to ask if the results in Theorem \ref{main}, Corollary \ref{coro:main}, Corollary \ref{coro2} hold for weights with exponential growth. For example, Theorem \ref{main} with $w=e^{-\lambda |x|},~\ld>0$ would give a strong pseudo-local property of gravity water waves. As we shall explain, the proof of our results can be divided into two parts: first, a study of the Dirichlet-Neumann operator in weighted spaces, and the second part makes use of a paradifferential machinery in weighted Sobolev spaces to paralinearize and symmetrize the system. For the first part, we are able to prove bound estimates for the Dirichlet-Neumann operator in the presence of "exponential weights" in the class $\mathcal{W}_{ex}$ (see Proposition \ref{coroetape2} below). However, for the (pseudo-) para-differential calculus, we have to restrict to "polynomial weights" in the class $\mathcal{W}_{po}$ due to the fact that, in general, the kernel of a pseudodifferential operator only decays polynomially (see paragraph \ref{explanation} 3. and the proof of Proposition \ref{pseudo}). 
\end{rema}
\begin{rema}
It  would of course be more satisfactory if the results could be formulated in terms of the derivatives of $\psi$ since $\psi$ is the trace of the velocity potential (on the free surface) and hence is determined up to additive constants. It should be possible to do so, modulo more technical complications; in particular, a Cauchy theory in Kato's spaces involving only regularity of $\psi$ in homogeneous spaces. We refer to a relating result of Lannes in \cite{LannesJAMS} and the references therein.
\end{rema}
\begin{rema}
Our proofs rely on the known Cauchy theories of Alazard-Burq-Zuily in \cite{ABZ1}, \cite{ABZ2}. To avoid a loss of $\mez$ derivatives, the authors assumed that initially, the trace of velocity $(B, V)$ are $\mez$ derivatives smoother than the natural threshold suggested by formula \eqref{BV} (see also Remark $1.4$, \cite{ABZ1}) . Another  way of avoiding this loss of derivatives can be found in Theorem $4.16$, \cite{LannesLivre} where instead of directly  imposing regularity condition on $\psi$, the author works with the "good unknown" $\underline\psi\thickapprox\psi-B\eta$.
\end{rema}
\subsection {Interpretation of the results}\label{explanation}
 1. The Zakharov system \eqref{ww} appears to be nonlocal, which comes from the fact that the Dirichlet-Neumann operator defined by \eqref{def:G} is nonlocal. This can be seen more concretely by considering the case of fluid domain with infinite depth (i.e. $\Gamma=\emptyset$) and free surface at rest (i.e. $\eta=0$). Then, the Dirichlet-Neumann operators is
$G(0)=|D_x|.$ However, Corollary \ref{coro:main} shows that the system is in fact still weakly local as explained below.\\
\hk  Take $s>1+{\frac{d}{2}}$. Let's restrict ourselves to a bounded set $\mathcal{A}$ of $\mathcal{H}_{ul}^s$ and suppose that we are observing a bounded domain, which by translation can be assumed to be centered at the origin, say $\mathcal{O}=B(0, 1)$. Let $U_{0,1},~U_{0,2}$  be two data in $\mathcal{A}$ such that they are identical  in  a ball $B(0, R)$  and have difference in $\mathcal{H}_{ul}^{s-1}$ of size $1$ outside this ball, where $R>1$ is a given distance. Take a "window" $\phi$ around our observation region $\mathcal{O}$, that is,  $\phi\in C^\infty_0(B(0, 3/2))$ and $\phi\equiv 1$ in  $\mathcal{O}$. Then by the estimate \eqref{eq.lipschitz1} we have for some $T=T(\mathcal{A})>0$ and any $N>0$
\begin{align*}
\Vert \phi(U_1-U_2)\Vert _{C([0, T], \mathcal{H}_{ul}^{s-1})}&\le C_N \Vert \langle \cdot\rangle^{-N} (U_1-U_2)\Vert _{C([0, T], \mathcal{H}_{ul}^{s-1})}\\
 &\le C_{N, \mathcal{A}} \Vert \langle \cdot\rangle^{-N} (U_{0,1}-U_{0,2})\Vert _{\mathcal{H}_{ul}^{s-1}}
\le  C_{N, \mathcal{A}} R^{-N}.
\end{align*}
\begin{center}
\begin{picture}(0,0)%
\includegraphics{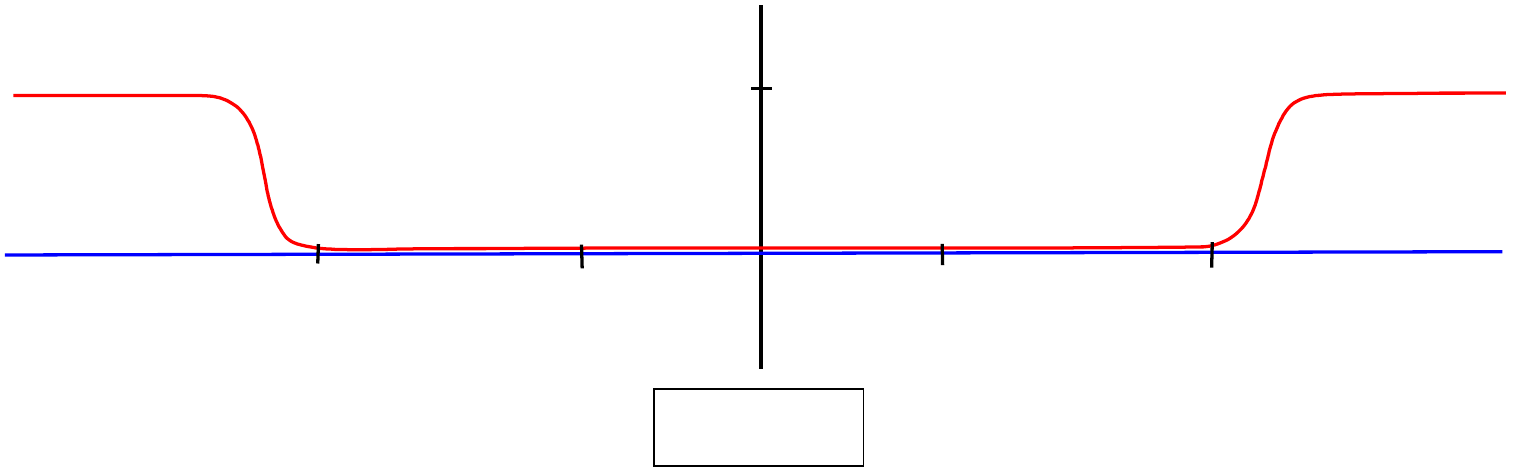}%
\end{picture}%
\setlength{\unitlength}{4144sp}%
\begingroup\makeatletter\ifx\SetFigFont\undefined%
\gdef\SetFigFont#1#2#3#4#5{%
  \reset@font\fontsize{#1}{#2pt}%
  \fontfamily{#3}\fontseries{#4}\fontshape{#5}%
  \selectfont}%
\fi\endgroup%
\begin{picture}(6907,2142)(77,-2451)
\put(2484,-1660){\makebox(0,0)[lb]{\smash{{\SetFigFont{12}{14.4}{\familydefault}{\mddefault}{\updefault}{\color[rgb]{0,0,0}$-1$}%
}}}}
\put(5416,-1655){\makebox(0,0)[lb]{\smash{{\SetFigFont{12}{14.4}{\familydefault}{\mddefault}{\updefault}{\color[rgb]{0,0,0}$R$}%
}}}}
\put(1220,-1666){\makebox(0,0)[lb]{\smash{{\SetFigFont{12}{14.4}{\familydefault}{\mddefault}{\updefault}{\color[rgb]{0,0,0}$-R$}%
}}}}
\put(4240,-1661){\makebox(0,0)[lb]{\smash{{\SetFigFont{12}{14.4}{\familydefault}{\mddefault}{\updefault}{\color[rgb]{0,0,0}$1$}%
}}}}
\put(3645,-662){\makebox(0,0)[lb]{\smash{{\SetFigFont{12}{14.4}{\familydefault}{\mddefault}{\updefault}{\color[rgb]{0,0,0}$1$}%
}}}}
\put(112,-1396){\makebox(0,0)[lb]{\smash{{\SetFigFont{12}{14.4}{\familydefault}{\mddefault}{\updefault}{\color[rgb]{0,0,0}$U_{0,1}$}%
}}}}
\put(3379,-2308){\makebox(0,0)[lb]{\smash{{\SetFigFont{12}{14.4}{\familydefault}{\mddefault}{\updefault}{\color[rgb]{0,0,0}$t=0$}%
}}}}
\put(133,-667){\makebox(0,0)[lb]{\smash{{\SetFigFont{12}{14.4}{\familydefault}{\mddefault}{\updefault}{\color[rgb]{0,0,0}$U_{0,2}$}%
}}}}
\end{picture}%
\end{center}
\begin{center}
\begin{picture}(0,0)%
\includegraphics{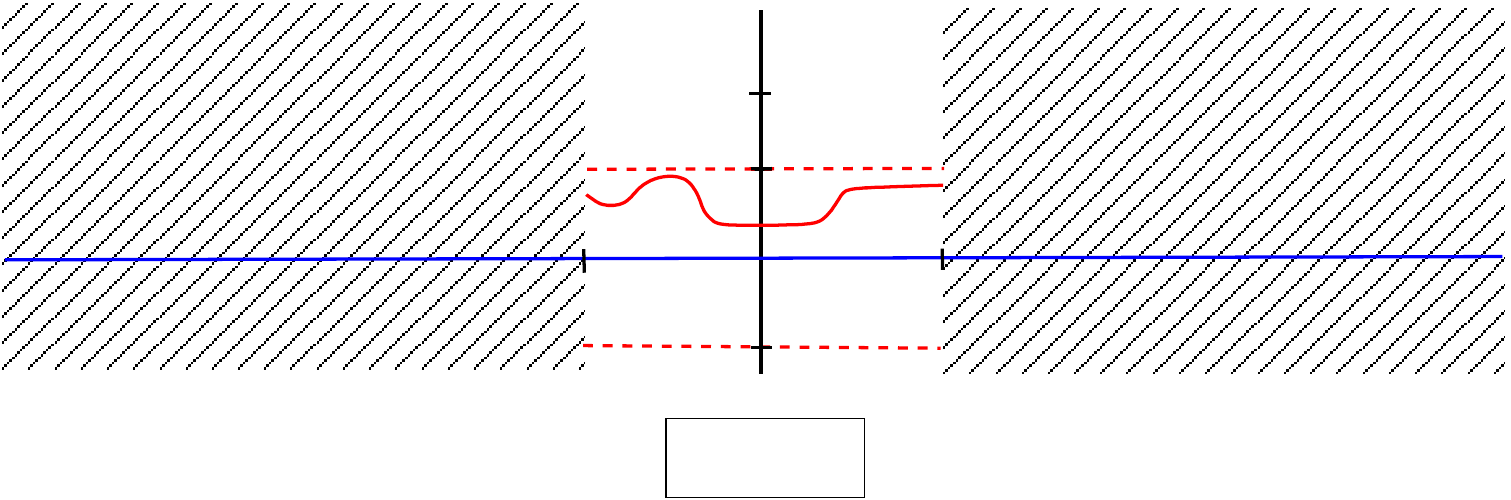}%
\end{picture}%
\setlength{\unitlength}{4144sp}%
\begingroup\makeatletter\ifx\SetFigFont\undefined%
\gdef\SetFigFont#1#2#3#4#5{%
  \reset@font\fontsize{#1}{#2pt}%
  \fontfamily{#3}\fontseries{#4}\fontshape{#5}%
  \selectfont}%
\fi\endgroup%
\begin{picture}(6891,2271)(77,-2571)
\put(4240,-1661){\makebox(0,0)[lb]{\smash{{\SetFigFont{12}{14.4}{\familydefault}{\mddefault}{\updefault}{\color[rgb]{0,0,0}$1$}%
}}}}
\put(3658,-964){\makebox(0,0)[lb]{\smash{{\SetFigFont{12}{14.4}{\familydefault}{\mddefault}{\updefault}{\color[rgb]{0,0,0}$\frac{1}{R^N}$}%
}}}}
\put(3038,-1271){\makebox(0,0)[lb]{\smash{{\SetFigFont{12}{14.4}{\familydefault}{\mddefault}{\updefault}{\color[rgb]{0,0,0}$U_2$}%
}}}}
\put(3635,-2046){\makebox(0,0)[lb]{\smash{{\SetFigFont{12}{14.4}{\familydefault}{\mddefault}{\updefault}{\color[rgb]{0,0,0}$-\frac{1}{R^N}$}%
}}}}
\put(3055,-1643){\makebox(0,0)[lb]{\smash{{\SetFigFont{12}{14.4}{\familydefault}{\mddefault}{\updefault}{\color[rgb]{0,0,0}$U_1$}%
}}}}
\put(3630,-687){\makebox(0,0)[lb]{\smash{{\SetFigFont{12}{14.4}{\familydefault}{\mddefault}{\updefault}{\color[rgb]{0,0,0}$1$}%
}}}}
\put(3172,-2453){\makebox(0,0)[lb]{\smash{{\SetFigFont{12}{14.4}{\familydefault}{\mddefault}{\updefault}{\color[rgb]{0,0,0}$0<t<T$}%
}}}}
\put(2762,-1665){\makebox(0,0)[lb]{\smash{{\SetFigFont{12}{14.4}{\familydefault}{\mddefault}{\updefault}{\color[rgb]{0,0,0}$-1$}%
}}}}
\end{picture}%
\end{center}
Therefore, under the dynamic governed by system \eqref{ww}, a difference of size $1$ outside the ball $B(0, R)$ of initial data leads to a difference of size $~R^{-N}$ of two solutions in the bounded domain $B(0, 1)$ (see the figures above). When $R\to +\infty$, the difference of two solutions tends to $0$ at a rate faster than any polynomial. In other words, to some extent, what happens  far away has small effect on a given bounded region; moreover this effect becomes smaller and smaller when the distance increases to $+\infty$. This gives us a weakly local property of gravity water waves. This property is indeed dictated by the polynomial decay off the diagonal of the kernel of differential operators in suitable classes, as we shall explain in point 3. below.\\
\hk 2. As a consequence of Corollary \ref{coro2}, the estimate \eqref{decay} with $\ld>0$ provides a spatial decay property for solutions. In classical Sobolev spaces, solutions always vanish at infinity. On the other hand, Theorem \ref{theo:ABZ} gives the existence of solutions in Kato's spaces which can be neither decaying nor periodic. The estimate \eqref{decay} however gives a conclusion for the intermediate case: as long as the datum decays at some rate which is at most algebraic 
in Kato's space, the solution decays also, and moreover, at the same rate.\\
\hk 3. Let us explain why the class $\mathcal{W}_{po}$ of weight that are at most polynomial growth  is a reasonable choice in our results. For this purpose, a good way is to look at the linearization of system \eqref{ww} around the rest state $(\eta, \psi)=(0, 0)$ (take $g=1$ and the flat bottom $\{y=-h,~h>0\}$) 
\bq\label{eq:line}
\left\{
\begin{aligned}
&\partial_t \eta-|D_x|\tanh(h|D_x|)\psi=0,\\
&\partial_t\psi+\eta=0
\end{aligned}
\right.
\eq
or equivalently, with $u:=\eta+i|D_x|^\mez\sqrt{\tanh(h|D_x|)} \psi$,
\[
\partial_t u+i|D_x|^\mez\sqrt{\tanh(h|D_x|)} u=0.
\]
Given a datum $u_0$ at time $t=0$, this linearized equation has the explicit solution 
\[
u(t,x)=p(t, D_x)u_0,
\]
where the symbol $p$ reads $p(t,\xi)=e^{-it|\xi|^\mez\sqrt{\tanh(h|\xi|)}}$. Then  for $w\in \mathcal{W}_{po}(\varrho),~\varrho\ge 0$ and $0<T<\infty$, we seek for the following  estimate
\bq\label{est:line}
\| w u\|_{C([0, T]; H^s_{ul})}\le C\| w u_0\|_{H^s_{ul}}.
\eq
Due to the presence of $\tanh(h|\xi|)$ we have that $p$ satisfies 
\bq\label{est:symbolp}
|\partial_\xi^\alpha p(t, \xi)|\le C_\alpha (1+\xi)^{-\mez|\alpha|},~\forall \alpha\in \xN^d,~(t, \xi)\in [0, T]\times \xR^d,
\eq
which is usually denoted by $p\in S^0_{\mez, 0}$. An adaptation of the proof of Proposition \ref{pseudo} then implies the estimate \eqref{est:line}. Indeed, for simplicity let us consider $s=0$. By writing $\chi_q=\chi_q\widetilde \chi_q$ (recall Notation \ref{widetildechi}) we need to show for any fixed $k\in \xZ^d$ 
\bq\label{intro:kernel}
A_k:=\sum_{q} w\chi_kp(t,D_x)\chi_qw^{-1}:{ L^2(\xR^d)\to L^2(\xR^d)} 
\eq
with norm bounded uniformly in $k$. Using the classical pseudo-differential theory, it suffices to prove \eqref{intro:kernel} for $q$ satisfying $|q-k|\ge M$ for fixed $M>0$ (cf  estimate \eqref{pseudo:near}). Due to the presence of $\chi_k$ it suffices to prove $A_k: L^2(\xR^d)\to L^\infty(\xR^d)$. To this end, we call 
\bq\label{intro:K}
K(x, y)=(2\pi)^{-d}\int_{\xR^d}e^{i(x-y)\xi}p(t,\xi)d\xi
\eq
 the kernel of the pseudo-differential operator $p(t, D_x)$ then the kernel of $A_k$ reads 
\[
H_k(x, y)=\sum_{|q-k|\ge M}w(x)\chi_k(x) K(x, y)\chi_q(y)w(y)^{-1}.
\]
 The Cauchy-Schwartz inequality implies
\[
\Vert A_k v\Vert_{L^\infty}\le \sup_x\Vert H_k(x, \cdot)\Vert_{L^2}\Vert v\Vert_{L^2}.
\]
By writing $\chi_q=\widetilde\chi_q\chi_q$ it suffices to show that $\Vert H_k\Vert_{L^\infty_{x, y}}$ is bounded by some constant independent of $k$. Indeed, remark that by choosing $M$ large enough,  on the support of $\chi_k(x)\chi_q(y)$ (in the expression of $H_k$) we have $|x-y|\ge \delta |k-q|$ for some $\delta>0$. Therefore, one can multiply both side of \eqref{intro:K} by $(x-y)^\gamma,~\gamma\in \xN^d$, integrate by parts and take into account the decay property (in $\xi$) \eqref{est:symbolp} of $p$ to derive
\bq\label{intro:sumq}
\vert H_k(x, y)\vert \le C_N\sum_{|q-k|\ge M}\frac{\vert\chi_k(x) w(x)\chi_q(y)w(y)^{-1}\vert}{\langle k-q\rangle^N},\quad\forall N\in\xN.
\eq
Observe that for any $w\in \mathcal{W}$ the absolute value of the numerator of each term in the above series is bounded by 
\[
Cw(k)w(q)^{-1}.
\]
Consequently, for the series in  \eqref{intro:sumq} to be convergent, it is reasonable to choose the weights that satisfy for some $\ld >0$
\[
w(k)w(q)^{-1}\les \langle k-q\rangle^\ld,~\forall k,~q\in \xZ^d.
\]
Then by choosing $N$ large enough the series in \eqref{intro:sumq} converges to some constant independent of $k$ as desired.\\
 This argument suggests heuristically that the finiteness of the fluid depth is likely to be necessary since otherwise, the symbol $p$ become $p(t,\xi)=|\xi|^{1/2}$ which is singular at $0$ and does not belong to a good class of symbols. The local property of the system is closely related to the finite propagation speed property. Indeed, the plane waves $u(t,x)=e^{i(x\cdot k-\omega(k) t)},\omega(k)=|k|^{1/2}$ are solutions to $u_t+i|D_x|^{1/2}u=0$. These plane waves propagate at (group) velocity $\nabla_k\omega(k)$, which is unbounded when the wavenumber $|k|$ tends to $0$. In contrast, for the finite depth case (say $h$), the dispersion relation reads $\omega(k)=\sqrt{|k|\tanh(h|k|)}$ and thus the speed of propagation is bounded over all wavenumber $k$. 
\begin{rema}
In the theory of pseudo-differential calculus, the terminology {\it pseudo-local} refers to the following property: if $T$ is a pseudo-differential operator then the singular support of $Tu$ is contained in the singular support of $u$. The proof of this result makes use of the fact that: the kernel of $T$ is $C^\infty$ off the diagonal $(x, x)$ in $\xR^d\times \xR^d$. This in turn stems from the decay property of the symbol of $T$.\\
The "pseudo-local property" in our result as explained above, also stems from the decay of the kernel of a pseudo-differential operator. However, such a decay is then translated not into the {\it regularity} (in term of the singular support) but the {\it persistence in weighted spaces}.
\end{rema}
\subsection{Plan of the proof} 
To prove Theorem \ref{main}  we follow essentially the scheme in \cite{ABZ2}. The first task is to adapt the paradifferential machinery to Kato's spaces with weights. This is done in Appendix \ref{paramachinery}, which can be of independent interest for other studies in this framework. Having this in hand, compare  to \cite{ABZ2} (and also \cite{ABZ1}) the main ingredient for the proof of Theorem \ref{main} reduces to the study of bound estimates, paralinearization and contraction estimate for the Dirichlet-Neumann operator. These are done in section \ref{descriptDN} and \ref{contractDN} below, respectively.\\
\\ \hk {\bf Acknowledgment.}~
This work was partially supported by the labex LMH through the grant no ANR-11-LABX-0056-LMH in the "Programme des Investissements d'Avenir". I would like to send my deepest thanks to my advisor, Prof. Nicolas Burq for his great guidance with many fruitful discussions and constant encouragement during this work. I sincerely thanks Prof. Claude Zuily for interesting discussions. Finally, I thanks the referees  for proposing many valuable suggestions that helped improve  both the content and the presentation of the manuscript. 
\section{A weighted description for the Dirichlet-Neumann operator}\label{descriptDN}
\begin{nota}
Throughout this paper, we denote $C>0$ and $\cF:\xR^+\to \xR^+$ are multiplicative constants and functions that may change from line to line within a proof. The notation $A\les B$ means that there exist $C>0$ such that $A\le C B$.
\end{nota}
\subsection{Definition of the Dirichlet-Neumann operator}
In this sections, we drop the time dependence of the domain and work on the domain of the form. 
 \begin{equation}\label{omega}
   \Omega =\{(x,y)\in \xR^{d+1}: \eta_*(x) <y < \eta(x)\}
 \end{equation}
where $\eta_*$ is a fixed bounded continuous function on $\xR^d$ and $\eta \in W^{1,\infty}(\xR^d)$. We assume that $\Omega$ contains a fixed strip
 \begin{equation}\label{condition}
      \Omega_h:= \{(x,y) \in \xR^{d+1} : \eta(x) -h \leq y <\eta(x) \}.
  \end{equation}
\subsubsection{Straightening the boundary} 
We recall here the change of variables introduced in \cite{ABZ1} (see section 3.1.1) to flatten the domain with free boundary (which is in turn inspired by Lannes \cite{LannesJAMS}). Consider the map $(x,z) \mapsto (x, \rho(x,z))$ from $\widetilde{\Omega}\defn\xR^d\times(-1,0)$ to $\Omega_h$ determined by
\begin{equation}\label{diffeo}
\rho(x,z)=  (1+z)e^{\delta z\langle D_x \rangle }\eta(x) -z\left [e^{-(1+ z)\delta\langle D_x \rangle }\eta(x) -h\right]\quad \text{if } (x,z)\in \widetilde{\Omega}.
\end{equation}
If $\eta\in W^{1,\infty}(\xR^d)$ and  $\delta$ is small enough this map is a~Lipschitz-diffeomorphism from~$ \widetilde{\Omega} $ to~$\Omega_h $ and moreover, $\partial_z\rho\ge c_0>0$ (see Lemma $3.6$, \cite{ABZ1}).
\begin{nota} For any function $f$ defined on $\Omega$, we  set
\begin{equation}\label{change}
 \widetilde{f}(x,z) = f(x,\rho(x,z))
\end{equation}
then 
\begin{equation}\label{lambda}
 \left\{
 \begin{aligned}
  \frac{\partial f}{\partial y}(x,\rho(x,z)) &= \frac{1}{\partial_z \rho}\partial_z \widetilde{f}(x,z):=\Lambda_1\widetilde{f}(x,z) \\
  \nabla_x f(x,\rho(x,z)) &= \big(\nabla_x \widetilde{f} - \frac{\nabla_x \rho}{\partial_z \rho}\partial_z \widetilde{f}\big)(x,z) := \Lambda_2 \widetilde{f}(x,z).
  \end{aligned}
  \right.
\end{equation}
\end{nota}
\subsubsection{Definition of the Dirichlet-Neumann operator $G(\eta)$}\label{defiDN} Let $\psi \in H_{ul}^\mez(\xR^d)$, we recall how $G(\eta)\psi$ is defined (section 3.1, \cite{ABZ2}). \\
For every $q \in \xZ^d$, set 
$\psi_q = \chi_q \psi \in H^\mez(\xR^d)$ then one can find $  \underline{ \psi}_q \in H^1(\Omega)$ such that $\underline{ \psi}_q  \arrowvert_{ y=\eta(x)}  = \psi_q(x)$ and  for some $\cF:\xR^+\to \xR^+$,
 \begin{align*} 
& (i) \quad \supp \underline{ \psi}_q \subset \{(x,y): \vert x-q \vert \leq 2, \eta(x) -h \leq y \leq \eta(x)\}\\
  & (ii) \quad \| \underline{ \psi}_q \|_{H^1( \Omega)} \leq \mathcal{F}(\Vert \eta \Vert_{W^{1,\infty}(\xR^d)}) \| \psi_q\|_{H^{\mez} ( \xR^d)}. 
   \end{align*}
Let $u_q\in H^{1,0}(\Omega):=\left\{v\in H^1(\Omega),~v\arrowvert_{\Sigma}=0\right\}$ be the unique variational solution, to equation $\Delta_{x,y} u_q = -\Delta_{x,y} \underline{\psi_q}$, which is characterized by 
\begin{equation}\label{eqvar}
\iint_{\Omega} \nabla_{x,y}u_q(x,y) \cdot \nabla_{x,y}\theta(x,y)dx dy = - \iint_{\Omega} \nabla_{x,y} \underline{\psi}_q(x,y) \cdot \nabla_{x,y}\theta(x,y)dx dy  
\end{equation}
for all $\theta \in H^{1,0}(\Omega)$.   The series $u:=\sum_{q\in \xZ^d} {u}_q$ is  then convergent in 
 $$ {H}^{1,0}_{ul}({\Omega}):=\left\{v:~\sup_{q\in \xZ^d}\lA \chi_qv\rA_{H^1(\Omega)}<+\infty~\text{and}~ v\arrowvert_{\Sigma}=0\right\}.$$
 Finally, $\Phi:=u+\underline{\psi}:=\sum_{q\in \xZ}u_q+\sum_{q\in \xZ}\underline{\psi}_q$ solves uniquely the elliptic problem 
\begin{equation}\label{eqPhi}
 \Delta_{x,y} \Phi = 0 \text{ in } \Omega, \quad  \Phi\arrowvert_{\Sigma} = \psi, \quad \frac{\partial \Phi}{\partial \nu}\arrowvert_ \Gamma = 0,
 \end{equation}
in the variational sense and moreover, there exists $\cF:\xR^+\to \xR^+$ such that
\[
\lA \Phi\rA_{H^1_{ul}(\Omega)}\le \cF(\lA \eta\rA_{W^{1,\infty}(\xR^d)})\lA \psi\rA_{H^{\mez}_{ul}(\xR^d)}.
\]
(see Proposition $3.3$, \cite{ABZ2})\\
 The Dirichlet-Neumann operator is defined by
\begin{equation}\label{G=}
\begin{aligned}
 G(\eta) \psi(x) &= (1+\vert \nabla_x \eta\vert^2)^\mez \frac{\partial \Phi}{\partial n}\arrowvert_\Sigma = \big(\frac{\partial \Phi}{\partial y} - \nabla_x \eta \cdot \nabla_x \Phi\big)\arrowvert_\Sigma\\
   &= \big( \Lambda_1\widetilde{\Phi}- \nabla_x \eta \cdot\Lambda_2 \widetilde{\Phi}\big)\arrowvert_{z=0} = \big( \Lambda_1\widetilde{\Phi}- \nabla_x \rho\cdot\Lambda_2 \widetilde{\Phi}\big)\arrowvert_{z=0}.
 \end{aligned}
 \end{equation}
\subsection {Elliptic regularity with weights}\label{elliptic}
We observe that if $u$ is a solution of the elliptic equation $\Delta u=0$ on $\Omega$ and $\widetilde{u} $ is its image via the diffeomorphism (\ref{diffeo}) then 
\[
(\Lambda_1^2+\Lambda^2_2)\widetilde{u} =0,
\]
which is equivalent to (see equation $(3.16)$, \cite{ABZ1})
  \begin{equation}\label{equ:modifie}
(\partial_z^2 + \alpha \Delta_x + \beta\cdot \nabla_x \partial_z - \gamma \partial_z) \widetilde{u} = 0,
\end{equation}
where 
\begin{equation}\label{alpha}
\alpha\defn \frac{(\partial_z\rho)^2}{1+|\partialx  \rho |^2},\quad 
\beta\defn  -2 \frac{\partial_z \rho \nabla_x \rho}{1+|\nabla_x  \rho |^2} ,\quad 
\gamma \defn \frac{1}{\partial_z\rho}\bigl(  \partial_z^2 \rho 
+\alpha\Delta_x \rho + \beta \cdot \nabla_x \partial_z \rho\bigr).
\end{equation}
These coefficients are estimated by
\begin{lemm}[\protect{\cite[Lemma~3.17]{ABZ2}}]\label{est-alpha} Let $J=(-1, 0)$ and $s>1+{\frac{d}{2}}$. There exists   $\mathcal{F}:\xR^+ \to \xR^+$ non decreasing such that (see Definition \ref{XY} for the definition of $X^{\mu}_{ul}$)
  \begin{equation*}
    \Vert \alpha \Vert_{X^{s-\mez}_{ul}(J)}  +  \Vert \beta \Vert_{X^{s-\mez}_{ul}(J)} + \Vert \gamma \Vert_{X^{s-\frac{3}{2}}_{ul}(J)} \\
   \leq \mathcal{F}\big(\Vert  \eta \Vert_{ H^{s+ \mez}_{ul}}\big).
        \end{equation*}
        \end{lemm}
Let us denote by $\mathcal{L}$ the linear differential operator
\bq\label{L}
\mathcal{L}=\partial_z^2 + \alpha \Delta_x + \beta\cdot \nabla_x \partial_z
\eq
and consider the following inhomogeneous initial value problem 
\begin{equation}\label{elliptique}
\left\{
\begin{aligned}
&(\mathcal{L}-\gamma\partial_z)\widetilde u 
= F  \quad \text{in } \, \xR^d \times  J, \\
&\widetilde{u} \arrowvert_{z=0} = \psi.  
\end{aligned}
\right.
\end{equation}
Recall Definition \ref{classW} for the definitions of weight classes $\mathcal{W},~\mathcal{W}_{po},~\mathcal{W}_{ex}$. It is clear that $\mathcal{W}_{po}(\varrho)\subset \mathcal{W}_{ex}(\varrho)$ for all $\varrho \ge 0$. For any $w\in \mathcal{W}$, defining 
\bq\label{defi:r2}
r_2:=\frac{\Delta w^{-1}}{w^{-1}},~r'_2:=\frac{\Delta w}{w},
\eq
we have that $r_2,~r'_2\in C^\infty_b(\xR^d)$. Now we fix a weight $w\in \mathcal{W}$ and set $\widetilde v=w\widetilde u$. A simple computation shows that  $\widetilde v$ satisfies 
\[
\mathcal{L}\widetilde v+(\beta\cdot r_1-\gamma)\partial_z\widetilde v+\alpha r_2\widetilde v +2\alpha r_1\cdot\nabla_x \widetilde v=wF.
\]
Next, set $\widetilde v_k=\chi_k\widetilde v$, then
\bq\label{Lvk}
\mathcal{L}\widetilde v_k=\chi_kwF+F_0+F_1
\eq
where
\[
\begin{cases}
F_0=\alpha\Delta\chi_k\widetilde v+2\alpha \nabla\chi_k\cdot\nabla_x\widetilde v+\beta\cdot\nabla_x\chi_k\partial_z\widetilde v-\chi_k\beta \cdot r_1\partial_z\widetilde v-\chi_k\alpha r_2\widetilde v -2\chi_k\alpha r_1\cdot\nabla_x\widetilde v,\\
F_1=\chi_k\gamma\partial_z\widetilde v.
\end{cases}
\]
Estimates for $F_j$'s are given in the next lemma.
\begin{lemm}\label{estF0F1}  Let $J=(-1, 0)$ and $s>1+{\frac{d}{2}}$. There exists  $\mathcal{F}:\xR^+ \to \xR^+$  non decreasing such that  for $-\mez\leq \sigma \leq s-1$ we have
 \bq\label{est:F0F1}\sum_{j=0}^1\Vert F_j \Vert_{Y^{\sigma + \mez}(J)} \leq    \mathcal{F}\big(\Vert  \eta \Vert_{H^{s+\mez}_{ul} }\bigr)\big(\lA w\psi\rA_{H^\sigma_{ul}}+ \Vert \nabla_{x,z} \widetilde{v}\Vert_{X^\sigma_{ul}(J)}\big).\eq
(see Definition \ref{XY} for the definition of $Y^{\mu}$), where $\mathcal{F}$ depends on $w$ only through the semi-norms of $r_i, r'_i,~i=1,2$ in $C^\infty_b(\xR^d)$.
\end{lemm}
\begin{proof}
1. It was proved in Lemma $3.20$, \cite{ABZ1} (applied with $\varepsilon=\mez$) that under the conditions of this lemma,
\[
\lA \gamma\partial_z\widetilde v\rA_{Y^{\sigma+\mez}(J)}\le C\lA \gamma\rA_{L^2(J, H^{s-1})}\lA \partial_z\widetilde v \rA_{X^\sigma(J)},
\]
whose proof uses only the regularity of $\gamma$ and $\partial_z \widetilde v$. By writing $\chi_k\gamma\partial_z\widetilde v=(\chi_k\gamma)(\widetilde\chi_k\partial_z\widetilde v)$ and using the proof of preceding estimate we obtain 
\begin{align*}
\lA F_1\rA_{Y^{\sigma+\mez}(J)}=\lA \chi_k\gamma\partial_z\widetilde v\rA_{Y^{\sigma+\mez}(J)}
&\le C\lA \gamma\rA_{L^2(J, H^{s-1}_{ul})}\lA \partial_z\widetilde v\rA_{X^\sigma_{ul}(J)}\\
&\le \mathcal{F}\big(\Vert  \eta \Vert_{H^{s+\mez}_{ul} }\bigr)\lA \partial_z\widetilde v\rA_{X^\sigma_{ul}(J)}.
\end{align*}
2. We turn to estimate $F_0$. All the terms containing either $\nabla_x\widetilde v$ or $\partial_z\widetilde v$ can be handled by the same method (remark that $r_1\in C^\infty_b(\xR^d)$). Let us consider for example $\alpha \nabla\chi_k\cdot\nabla_x\widetilde v$. There exists $M>0$ such that if $|k-j|>M$ then $\supp\chi_k\cap\supp\chi_j=\emptyset$. Therefore, it suffices to estimate $A=\chi_j\alpha \nabla\chi_k\cdot\nabla_x\widetilde v$ for $|j-k|\le M$.  We have the following product rule in Sobolev spaces (see for instance, Corollary $2.11$ $(i)$, \cite{ABZ1}): if $s_0\le s_1,~s_0\le s_2$,~$s_1+s_2>0$ and $s_0<s_1+s_2-{\frac{d}{2}}$ then there exists $C>0$ such that for all $u_1\in H^{s_1},~u_2\in H^{s_2}$ there holds
\bq\label{product:rule}
\lA u_1u_2\rA_{H^{s_0}}\le C\lA u_1\rA_{H^{s_1}}\lA u_2\rA_{H^{s_2}}.
\eq
The preceding result applied with $s_0=\sigma,~s_1=s,~s_2=\sigma$  together with Lemma \ref{est-alpha} leads to 
\[
\lA A\rA_{L^2(J, H^\sigma)}\le \lA \chi_k\alpha\rA_{L^2(J, H^s)}\lA \nabla\chi_k\cdot\nabla_x\widetilde v\rA_{L^\infty(J, H^\sigma)}\le \mathcal{F}\big(\Vert  \eta \Vert_{H^{s+\mez}_{ul} }\bigr)\lA \nabla_x \widetilde v\rA_{X^\sigma_{ul}}.
\]
3. We are left with two terms $\alpha\Delta\chi_k\widetilde v$ and $\chi_k \alpha r_2\widetilde v$, which can be treated in the same way (remark that $r_2\in C^\infty_b(\xR^d)$). Let us consider for example $\alpha\Delta\chi_k\widetilde v$. As in 2., one only need to estimate $\chi_j\alpha\Delta\chi_k\widetilde v$ for $j$ close to $k$. The product rule \eqref{product:rule} gives
\begin{align*}
\lA \chi_j\alpha\Delta\chi_k\widetilde v\rA_{Y^{s+\mez}(J)}&\le \lA \chi_j\alpha\Delta\chi_k\widetilde v\rA_{L^2(J, H^\sigma)}\\
&\le C\lA \alpha\Delta\chi_k\rA_{L^2(J, H^s)}\lA \chi_j\widetilde v\rA_{L^\infty(J, H^\sigma)}\\
&\le  C\lA \alpha\rA_{L^2(J, H^s_{ul})}\lA \chi_j\widetilde v\rA_{L^\infty(J, H^\sigma)}.
\end{align*}
Now, by writing 
\bq\label{NL}
\chi_j\widetilde v(x,z)= \chi_j\widetilde v(x, 0)+\int_0^z  \chi_j\partial_z\widetilde v(x, \tau)d\tau= \chi_jw\psi(x)+\int_0^z \chi_j\partial_z\widetilde v(x, \tau)d\tau,
\eq
we obtain
\begin{align*}
\lA \chi_j\widetilde v\rA_{L^\infty(J, H^\sigma)}\le \lA w\psi\rA_{H^\sigma_{ul}}+\lA \partial_z\widetilde v\rA_{L^\infty(J, H^\sigma)_{ul}}\le \lA w\psi\rA_{H^\sigma_{ul}}+\lA \partial_z\widetilde v\rA_{X^\sigma_{ul}(J)}.
\end{align*}
Consequently,
\[
\lA \chi_j\alpha\Delta\chi_k\widetilde v\rA_{Y^{\sigma+\mez}(J)}\le \cF(\lA \eta\rA_{H^{s+\mez}_{ul}})\big( \lA w\psi\rA_{H^\sigma_{ul}}+\lA \partial_z\widetilde v\rA_{X^\sigma_{ul}(J)}\big).
\]
\end{proof}
\begin{rema}\label{rema:s0}
Lemma \ref{estF0F1} is in the same spirit of Lemma $3.18$, \cite{ABZ2}. However, in Lemma $3.18$, \cite{ABZ2} the authors considered two cases corresponding to two  ranges of $\sigma$: $-\mez\le \sigma<s-\tdm$ and $s-\tdm\le \sigma\le s-1$. This aimed to keep in the estimate \eqref{est:F0F1} the function $\cF$ depending only on $\|\eta\|_{H^{s_0+\mez}_{ul}}$ for any $1+\frac{d}{2}<s_0\le s$, which appeared in their finial {\it a priori} estimate (see Proposition $4.7$, \cite{ABZ2}).  Here, however, for our contraction estimates we do not need this tame estimate. In fact, our contraction estimates shall be established in $1$-derivative lower Sobolev spaces, hence we even do not use weighted bounds in the highest norms. 
\end{rema}
Next, we prove an elliptic regularity theorem with weights for $\nabla_{x,z}\widetilde v$:
\begin{theo}\label{regell}
Let $J=(-1, 0)$, $s>1+{\frac{d}{2}}$ and $w\in \mathcal{W}_{ex}(\varrho),~\varrho\ge 0$. Let $\widetilde{u}$  be a solution of the problem (\ref{elliptique}) and set  $\widetilde{v}=w\widetilde{u}$. For $-\mez \leq \sigma \leq s-\mez$ let  $ \eta  \in H^{s + \mez}_{ul}(\xR^d)$ 
satisfying~\eqref{condition},  $w\psi \in H^{\sigma+1}_{ul}(\xR^d), F \in Y_{ul}^{\sigma}(J)$  and 
\begin{equation}\label{-mez-i}
\Vert \nabla_{x,z}\widetilde{v}\Vert_{X_{ul}^{-\mez}(J)}  <+\infty.
\end{equation}
Then for every $z_0\in ]-1,0[$ there exists $\mathcal{F} :\xR^+ \to \xR^+ $ non decreasing, 
depending only on $(s, d)$ and the semi-norms of $r_i,~r_i', i=1,2$ (in $C^\infty_b(\xR^d)$) such that\\
\bq\label{elliptic:1}
\Vert \nabla_{x,z}\widetilde{v} \Vert_{X^\sigma_{ul}(z_0,0)} 
\leq \mathcal{F} \bigl(\Vert  \eta  \Vert_{H^{s+ \mez}_{ul}}\bigr) 
\Big\{ \Vert w\psi \Vert_{H^{\sigma+1}_{ul} } +   \Vert wF \Vert_{Y_{ul}^{\sigma}(J)}  
+  \Vert \nabla_{x,z}\widetilde{v}\Vert_{X_{ul}^{-\mez}(J)} \Big\}.
\eq
Consequently,
\bq\label{elliptic:2}
\Vert w\nabla_{x,z}\widetilde{u} \Vert_{X^\sigma_{ul}(z_0,0)} 
\leq \mathcal{F} \bigl(\Vert  \eta  \Vert_{H^{s+ \mez}_{ul}}\bigr) 
\Big\{ \Vert w\psi \Vert_{H^{\sigma+1}_{ul} } +   \Vert wF \Vert_{Y_{ul}^{\sigma}(J)}  
+  \Vert w\nabla_{x,z}\widetilde{u}\Vert_{X_{ul}^{-\mez}(J)} \Big\}.
\eq
   \end{theo}
\begin{proof}
Estimate \eqref{elliptic:1} is a  consequence of Proposition $3.19$, Proposition $3.20$ and the proof of Proposition $3.16$ in \cite{ABZ2}, tacking into account Remark \ref{rema:s0}. We now derive \eqref{elliptic:2} using \eqref{elliptic:1}. Remark first that $\partial_z\widetilde v=w\partial_z\widetilde u$. Next, we write 
\[
w\nabla_x\widetilde u=\nabla_x\widetilde v-\widetilde u\nabla_xw,
\]
where 
\[
\widetilde u(x,z)\nabla_xw(x)=r'_1(x)w(x)\widetilde u(x, z)=r'_1(x)w(x)\Big( \widetilde u(x, 0)+\int_0^z\partial_z\widetilde u(x, \tau)d\tau\Big),
\]
which implies (using again $r'_1\in C^\infty_b(\xR^d)$)
\[
\lA \widetilde u\nabla_xw\rA_{X^\sigma_{ul}}\le C\big( \lA w\psi\rA_{H^{\sigma+\mez}_{ul}}+\lA w\partial_z \widetilde u\rA_{X^\sigma_{ul}}\big)\le C\big( \lA w\psi\rA_{H^{\sigma+\mez}_{ul}}+\lA \partial_z \widetilde v\rA_{X^\sigma_{ul}}\big).
\]
We have proved that 
\[
\lA w\nabla_{x,z}\widetilde u\rA_{X^\sigma_{ul}}\le C\big(\lA w\psi\rA_{H^{\sigma+\mez}_{ul}}+\lA \nabla_{x,z} \widetilde v\rA_{X^\sigma_{ul}}\big).
\]
Likewise, it holds that 
\[
\lA \nabla_{x,z}\widetilde v\rA_{X^{-\mez}_{ul}}\le C\big(\lA w\psi\rA_{H^0_{ul}}+\lA w\nabla_{x,z} \widetilde u\rA_{X^{-\mez}_{ul}}\big).
\]
The two inequalities above show that \eqref{elliptic:2} is a consequence of \eqref{elliptic:1} (notice that $\sigma+1\ge 0$).
\end{proof}
\begin{rema}\label{remark: -mez}
 We remark that in all the results stated below, the function $\mathcal{F}$  depends on $w$ only through the semi-norms of $r_i$ and $r'_i,~i=1,2$ in $C^\infty_b(\xR^d)$.
\end{rema}
To apply Theorem \ref{regell} we need a base estimate in the low norm $X^{-\mez}_{ul}$. For the proof of this, let us recall a classical interpolation result.
\begin{lemm}[\protect{\cite[Theorem~3.1]{Lions}}]\label{interpo}
Let $J = (-1,0)$ and $\sigma\in \xR$.  Let  $f \in L_z^2(J, H^{\sigma+ \mez}(\xR^d))$ be such that $\partial_z f \in L_z^2(J, H^{\sigma-\mez}(\xR^d))$. Then $f \in C_z^0([-1,0], H^{\sigma}(\xR^d))$ and there exists an absolute constant $C>0$ such that
$$
\Vert f\Vert_{C_z^0([-1,0], H^{\sigma}(\xR^d))} \leq C
\Vert f \Vert _{L_z^2(J, H^{\sigma+ \mez}(\xR^d))} +C\Vert \partial_zf \Vert_{L_z^2(J, H^{\sigma- \mez}(\xR^d))}.
$$
 \end{lemm}
Recall also here the Poincar\'e inequality proved in \cite{ABZ2} (cf. Remark $3.2$) for fluid domains with finite depth of type $\Omega$ (cf. \eqref{defi:domain}). 
\begin{lemm}\label{lemm:Poincare}
Let $\theta,~\theta_*\in C^0_b(\xR^d)$ satisfying $\|\theta-\theta_*\|_{L^\infty(\xR^d)}>0$. Define
\[
\mathcal{O}=\{(x, y)\in \xR^d\times \xR: \theta_*(x)\le y\le \theta(x)\}
\]
and 
\[
H^{1,0}(\mathcal{O})=\{u\in L^2(\mathcal{O}): \nabla_{x,y}u\in L^2(\mathcal{O}) ~\text{and}~u\arrowvert_{y=\theta(x)}=0\}.
\]
Then for all $u\in H^{1,0}(\mathcal{O}),~\alpha\in C^\infty_b(\xR^d),~\alpha\ge 0$, there holds
\bq\label{Poincare}
\iint_{\mathcal{O}}\alpha(x)|u(x, y)|^2dxdy\le \Vert \theta-\theta_*\Vert_{L^\infty(\xR^d)}^2\iint_{\mathcal{O}}\alpha(x)|\partial_yu(x,y)|^2dxdy.
\eq
\end{lemm}
\begin{rema}
In Remark $3.2$, \cite{ABZ2} the constant appearing in the Poincar\'e inequality is stated to be dependent only on $\|\theta\|_{L^\infty}+\|\theta_*\|_{L^\infty}$. However, it is easy to track the proof to derive the explicit constant $\Vert \theta-\theta_*\Vert_{L^\infty(\xR^d)}^2$ in \eqref{Poincare}.
\end{rema}
\begin{prop}\label{rec00}
Let $J=(-1,0 )$, $s>1+{\frac{d}{2}}$. Let $\Phi$ be the unique solution to (\ref{eqPhi}). Then the following statements hold true.\\
$(i)$ For every $w\in \mathcal{W}_{po}(\varrho),~\varrho\ge 0$, one can find a non decreasing function $\mathcal{F}: \xR^+ \to \xR^+$ such that
\[
 \Vert w\nabla_{x,z}\widetilde{\Phi}\Vert_{X^{-\mez}_{ul}(J)} \leq \mathcal{F}(\Vert \eta \Vert_{H^{s+\mez}_{ul}(\xR^d)}) \Vert  w\psi \Vert_{H^\mez_{ul}(\xR^d)}.
\]
$(ii)$ There exists an absolute constant $C_*=C_*(d)>0$ such that for all $w\in \mathcal{W}_{ex}(\varrho)$ and $\varrho\le \varrho_*$ with 
\bq\label{defi:varrho*}
\varrho_*C_*\lA \eta-\eta_*\rA^2_{L^\infty(\xR^d)}=\mez,
\eq
 one can find a non decreasing function $\mathcal{F}: \xR^+ \to \xR^+$ such that
\[
 \Vert w\nabla_{x,z}\widetilde{\Phi}\Vert_{X^{-\mez}_{ul}(J)} \leq \mathcal{F}(\Vert \eta \Vert_{H^{s+\mez}_{ul}(\xR^d)}) \Vert  w\psi \Vert_{H^\mez_{ul}(\xR^d)}.
\]
   \end{prop}
\begin{proof} We proceed in two steps.\\
{\it Step 1.}~ By Lemma 3.6, \cite{ABZ2} one can find an absolute constant $C_*>0$ such that for all $\mu>0$ satisfying
\bq
\mu C_*\lA \eta-\eta_*\rA^2_{L^\infty(\xR^d)}\le 1
\label{delta1}
\eq
there exists $\mathcal{F}:\xR^+\to \xR^+$ non decreasing such that for all $q\in \xZ^d$ we have
\bq\label{uq:e}
\|e^{\mu\langle x-q\rangle}\nabla_{x,y}u_q\|_{L^2(\Omega)}\le  \mathcal{F}(\|\eta\|_{W^{1,\infty}(\xR^d)})\|\psi_q\|_{H^{\frac{1}{2}}(\xR^d)}.
\eq
Using properties $(i)$ and $(ii)$ above of $\underline \psi_q$ (see section \ref{defiDN}), we see that (\ref{uq:e}) also holds for $u_q$ replaced by $\underline \psi_q$ for any $\mu>0$ and thus (\ref{uq:e}) is true for $u_q$ replaced by $\Phi_q=u_q+\underline \psi_q$, i.e., 
\[
\|e^{\mu\langle x-q\rangle}\nabla_{x,z}\Phi_q\|_{L^2(\Omega)}\le  \mathcal{F}(\|\eta\|_{W^{1,\infty}(\xR^d)})\|\psi_q\|_{H^{\frac{1}{2}}(\xR^d)}
\]
for any $\mu>0$ satisfying (\ref{delta1}).\\
Using the diffeomorphism (\ref{diffeo}) we deduce that
\bq\label{phiqw}
\|e^{\mu\langle x-q\rangle}\nabla_{x,z}\widetilde\Phi_q\|_{L^2(J, L^2(\xR^d))}\le  \mathcal{F}(\|\eta\|_{W^{1,\infty}(\xR^d)})\|\psi_q\|_{H^{\frac{1}{2}}(\xR^d)}.
\eq
Consider a weight $w$ that is a priori in $\mathcal{W}$. On the support of $\chi_k$, we have $w(x)e^{\mu\langle x-q\rangle}\sim w(k)e^{\mu\langle k-q\rangle}$. Hence by the product rule \eqref{product:rule} we have
\bq\label{phiqw:1}
\|\chi_kw\nabla_{x,z}\widetilde\Phi_q\|_{L^2(J, L^2(\xR^d))}\le Cw(k)e^{-\mu\langle k-q\rangle} \mathcal{F}(\|\eta\|_{W^{1,\infty}(\xR^d)})\|\psi_q\|_{H^{\frac{1}{2}}(\xR^d)},
\eq
from which it follows that 
\bq\label{sum:Phi}
 \begin{aligned}
\|\chi_k w\nabla_{x,z}\widetilde\Phi\|_{L^2(J, L^2(\xR^d))}&\le \sum_q\|\chi_k w\nabla_{x,z}\widetilde\Phi_q\|_{L^2(J, L^2(\xR^d))}\\
&\le C\sum_qw(k)e^{-\mu\langle k-q\rangle} \mathcal{F}(\|\eta\|_{W^{1,\infty}(\xR^d)})\|\psi_q\|_{H^{\frac{1}{2}}(\xR^d)}\\
& \le C\sum_qw(k)w^{-1}(q)e^{-\mu\langle k-q\rangle} \mathcal{F}(\|\eta\|_{W^{1,\infty}(\xR^d)})\|w\psi_q\|_{H^{\frac{1}{2}}(\xR^d)}.
 \end{aligned}
\eq
Now we distinguish two cases: \\
$(i)$ $w\in \mathcal{W}_{po}(\varrho),~\varrho\ge 0$. By definition, $w(k)w^{-1}(q)\le C\langle k-q\rangle^\varrho$ and thus the final sum in \eqref{sum:Phi} converges for any $\rho\ge 0$, which leads to 
\bq\label{step1}
\|\chi_k w\nabla_{x,z}\widetilde\Phi\|_{L^2(J, L^2(\xR^d))}\le  \mathcal{F}(\|\eta\|_{W^{1,\infty}(\xR^d)})\|w\psi\|_{H^{\frac{1}{2}}_{ul}(\xR^d)}.
\eq
$(ii)$ $w\in \mathcal{W}_{ex}(\varrho),~\varrho\ge 0$. Choosing $\mu,~\varrho_*$ such that 
\[
\mu C_*\lA \eta-\eta_*\rA_{L^\infty(\xR^d)}^2=1, ~\varrho_* C_*\lA \eta-\eta_*\rA_{L^\infty(\xR^d)}^2=\mez
\]
then the final sum in \eqref{sum:Phi} converges for all $\varrho\le \varrho_*$ and one also ends up with \eqref{step1}.\\
{\it Step 2.} Let us fix a weight $w$ as in $(i)$ or $(ii)$. To complete the proof of this lemma, it remains to show for any $k\in \xZ^d$ that
\bq\label{step2}
\Vert \chi_k w\nabla_{x,z}\widetilde\Phi\Vert_{L^{\infty}(J, H^{-\mez}(\xR^d))}\le  \mathcal{F}(\|\eta\|_{H^{s+\mez}_{ul}})\|w\psi\|_{H^{\frac{1}{2}}_{ul}(\xR^d)}.
\eq
By the interpolation Lemma \ref{interpo}
\[
\Vert w\chi_k\nabla_{x}\widetilde{\Phi}\Vert_{L^{\infty}(J, H^{-\frac{1}{2}})}\le  \|w\chi_k\nabla_{x}\widetilde{\Phi}\|_{L^2(J, L^2)}+\|w\chi_k\partial_z\nabla_{x}\widetilde{\Phi}\|_{L^2(J, H^{-1})}.
\]
The first term on the right-hand side is estimated by \eqref{step1}, so we need to estimate 
\[ M:=\|w\chi_k\nabla_{x}\partial_z\widetilde{\Phi}\|_{L^2(J, H^{-1})}.
\]
Notice that for any acceptable weight $\omega\in \mathcal{W}$, there holds with $\widetilde \chi\in C_0^{\infty}(\xR^d)$ and $\widetilde \chi=1$ on $\supp\chi$ that
 \bqa
\|\omega\chi_k\nabla_xf\|_{H^s(\xR^d)}&\le & \|\nabla(\omega\chi_kf)\|_{H^s(\xR^d)}+\|\omega\nabla_x\chi_kf\|_{H^s(\xR)^d}+\|\nabla_x\omega\chi_kf\|_{H^s(\xR^d)}\\
&\le & \|\chi_k\omega f\|_{H^{s+1}(\xR^d)}+\|\omega\nabla_x\chi_kf\|_{H^s(\xR^d)}+ \Vert r'_1\widetilde\chi_k\|_{H^s}\|\chi_k\omega f\|_{H^s(\xR^d)}
\eqa
where $r'_1=\frac{\nabla \omega}{\omega}$ as in Definition \ref{classW}. This implies 
\bq\label{Poincare:w}
\|\omega\chi_k\nabla_xf\|_{H^s(\xR^d)}\le C\|\omega f\|_{H^{s+1}_{ul}(\xR^d)},\quad \forall \omega\in \mathcal{W},~s\in \xR.
\eq
Applying this estimate and \eqref{step1} leads to
\[
M\le C\|w\partial_z\widetilde{\Phi}\|_{L^2(J, L^2)_{ul}}\le \mathcal{F}(\|\eta\|_{H^{s+\frac{1}{2}}_{ul}}) \|w\psi\|_{H^{\frac{1}{2}}_{ul}}.
\]
Finally, to obtain \eqref{step2} we shall prove
 \bq
\|w\partial_{z}\widetilde{\Phi}\|_{L^{\infty}(J, H^{-\frac{1}{2}})_{ul}}\le \mathcal{F}(\|\eta\|_{H^{s+\frac{1}{2}}_{ul}}) \|w\psi\|_{H^{\frac{1}{2}}_{ul}}.
\label{phiqw3}
\eq
Again, by interpolation, 
\[
\|w\chi_k\partial_z\widetilde{\Phi}\|_{L^{\infty}(J, H^{-\frac{1}{2}})}\le  \|w\chi_k\partial_z\widetilde{\Phi}\|_{L^2(J, L^2)}+\|w\chi_k\partial_z^2\widetilde{\Phi}\|_{L^2(J, H^{-1})}.
\]
It remains to estimate $A:=\|w\chi_k\partial_z^2\widetilde{\Phi}\|_{L^2(J, H^{-1})}$. Taking into account the fact that $\widetilde\Phi_q$ satisfies equation (\ref{equ:modifie}), we have
\[
A\le \sum_q A_{1,q}+A_{2,q}+A_{3,q},
\]
where by the product rule \eqref{product:rule} (remark that $s>1+\frac{d}{2}$ is sufficient), Lemma \ref{est-alpha} and \eqref{step1}, 
   \begin{equation*}
  \begin{aligned}
  A_{1,q} &=\Vert \chi_k w\alpha \Delta \widetilde{\Phi}_q\Vert_{L^2(J, H^{-1} )}\leq \Vert \alpha \Vert_{L^\infty(J,H^{s-\mez} )_{ul}} \Vert w\Delta \widetilde{\Phi}_q\Vert_{L^2(J, H^{-1}  )_{ul}},\\
  A_{2,q} &=  \Vert \chi_k w\beta \partial_z \nabla_x  \widetilde{\Phi}_q\Vert_{L^2(J, H^{-1} )}\leq \Vert \beta \Vert_{L^\infty(J,H^{s-\mez} )_{ul}} \Vert w\partial_z \nabla_x  \widetilde{\Phi}_q\Vert_{L^2(J, H^{-1}  )_{ul}},\\
  A_{3,q} &=  \Vert \chi_k w\gamma\partial_z   \widetilde{\Phi}_q\Vert_{L^2(J, H^{-1} )}\leq \Vert \gamma \Vert_{L^\infty(J,H^{s-\frac{3}{2}}) _{ul}}\Vert w\partial_z \widetilde{\Phi}_q \Vert_{L^2(J,L^2 )_{ul}}.
  \end{aligned}
\end{equation*}
Finally, to sum $\Vert w\nabla_{x,z} \widetilde{\Phi}_q \Vert_{L^2(J,L^2 )_{ul}}$ over $q\in \xZ^d$, one makes use of \eqref{phiqw:1} and argues as in \eqref{sum:Phi}. The proof of Proposition \ref{rec00} is complete.
\end{proof}
\begin{rema}
In statement $(ii)$ above, the function $\cF$ depends on $\varrho$, which is in turn bounded from above by $C\| \eta-\eta_*\|_{L^\infty}^{-2}$. Therefore, $\cF$ is really increasing in $\|\eta\|_{H^{s+\mez}_{ul}}$ if the fluid depth $ \| \eta-\eta_*\|_{L^\infty}$ is bounded from below by some positive constant.
\end{rema}
Using Proposition~\ref{rec00} as the ground step for the regularity Theorem \ref{regell} we now prove a weighted estimate for $\widetilde\Phi$ and its gradient.
 \begin{coro}\label{coroetape1}
  Let $J=(-1,0 )$, $s>1+{\frac{d}{2}}$. Let $\Phi$ be the unique solution to (\ref{eqPhi}). Then the following statements hold true.\\
$(i)$ For every $w\in \mathcal{W}_{po}(\varrho),~\varrho\ge 0$ and $-\mez\le \sigma\le s-\mez$, one can find a non decreasing function $\mathcal{F}: \xR^+ \to \xR^+$ such that  for any $z_0\in (-1, 0)$,
 $$\Vert  w\widetilde{\Phi}  \Vert_{X^{\sigma+1}_{ul}(z_0,0)}+  \Vert w \nabla_{x,z}\widetilde{\Phi} \Vert_{X^\sigma_{ul}(z_0,0)}  \leq \mathcal{F}\bigl(\Vert   \eta  \Vert_{H^{s+ \mez}_{ul}}\bigr)     \Vert w\psi \Vert_{H^{\sigma+1}_{ul}}.
$$
$(ii)$ For every $w\in \mathcal{W}_{ex}(\varrho)$ with $\varrho\le \varrho_*$ (defined by \eqref{defi:varrho*}) and $-\mez\le \sigma\le s-\mez$, one can find a non decreasing function $\mathcal{F}: \xR^+ \to \xR^+$ such that  for any $z_0\in (-1, 0)$,
 $$\Vert  w\widetilde{\Phi}  \Vert_{X^{\sigma+1}_{ul}(z_0,0)}+  \Vert w \nabla_{x,z}\widetilde{\Phi} \Vert_{X^\sigma_{ul}(z_0,0)}  \leq \mathcal{F}\bigl(\Vert   \eta  \Vert_{H^{s+ \mez}_{ul}}\bigr)     \Vert w\psi \Vert_{H^{\sigma+1}_{ul}}.
$$
\end{coro}
\begin{proof}
Observe that $\widetilde{\Phi}$ satisfies~\eqref{elliptique} with $F=0$. According to Proposition~\ref{rec00}, with the weight $w$ given either in $(i)$ or $(ii)$ we have 
\begin{equation*}
\Vert w\nabla_{x,z}\widetilde{\Phi}\Vert_{X^{-\mez}_{ul}(z_0,0)} \leq \mathcal{F}(\Vert \eta \Vert_{ H^{s_0+ \mez}_{ul}}) \Vert w\psi \Vert_{H^\mez_{ul} }<+\infty.
\end{equation*}
Theorem \ref{regell} then leads to the desired estimate for $\Vert w \nabla_{x,z}\widetilde{\Phi} \Vert_{X^\sigma_{ul}(z_0,0)}$. Consequently, the argument in \eqref{NL} leads to
\[
\Vert w\widetilde{\Phi}\Vert_{X^\sigma_{ul}(z_0,0)}\le \lA w\psi\rA_{H^{\sigma+\mez}}+ \Vert w\partial_z\widetilde{\Phi}\Vert_{X^\sigma_{ul}(z_0,0)}\le  \mathcal{F}\bigl(\Vert   \eta  \Vert_{H^{s+ \mez}_{ul}}\bigr)     \Vert w\psi \Vert_{H^{\sigma+1}_{ul}}.
\]
Finally, using the fact that for any $W\in \mathcal{W}$, 
\[
\Vert Wu\Vert_{H^{\sigma+1}_{ul}}\le C\big(\Vert W\nabla u\Vert _{H^{\sigma}_{ul}} +\Vert Wu\Vert _{H^{\sigma}_{ul}}\big)
\]
we derive 
\[
\Vert  w\widetilde{\Phi}  \Vert_{X^{\sigma+1}_{ul}(z_0,0)}\le  \mathcal{F}\bigl(\Vert   \eta  \Vert_{H^{s+ \mez}_{ul}}\bigr)     \Vert w\psi \Vert_{H^{\sigma+1}_{ul}}.
\]
\end{proof}
 Corollary \ref{coroetape1} implies  the following weighted estimate for the Dirichlet-Neumann operator, which is of independent interest.
 \begin{prop}\label{coroetape2}
Let $s>1+{\frac{d}{2}}$ and $\eta\in H^{s+\mez}(\xR^d)_{ul}$. Then the following statements hold true.\\
$(i)$ For every $w\in \mathcal{W}_{po}(\varrho),~\varrho\ge 0$ and $-\mez\le \sigma\le s-\mez$, one can find a non decreasing function $\mathcal{F}: \xR^+ \to \xR^+$ such that 
   $$\Vert w G(\eta)\psi \Vert_{H^\sigma_{ul}}  \leq \mathcal{F}\bigl(\Vert   \eta  \Vert_{H^{s+ \mez}_{ul}}\bigr)     \Vert w\psi \Vert_{H^{\sigma+1}_{ul}}.$$  
$(ii)$ For every $w\in \mathcal{W}_{ex}(\varrho)$ with $\varrho\le \varrho_*$ (defined by \eqref{defi:varrho*}) and $-\mez\le \sigma\le s-\mez$, one can find a non decreasing function $\mathcal{F}: \xR^+ \to \xR^+$ such that
   $$\Vert wG(\eta)\psi \Vert_{H^\sigma_{ul}}  \leq \mathcal{F}\bigl(\Vert   \eta  \Vert_{H^{s+ \mez}_{ul}}\bigr)     \Vert w\psi \Vert_{H^{\sigma+1}_{ul}}.$$  
\end{prop}
\begin{proof}
Let $w$ be the weight as in $(i)$ or $(ii)$. By \eqref{G=},
\[
G(\eta)\psi=\big( \Lambda_1\widetilde{\Phi}- \nabla_x \rho\cdot\Lambda_2 \widetilde{\Phi}\big)\arrowvert_{z=0}=:H\arrowvert_{z=0}.
\]
Owing to Lemma \ref{interpo}, we have for any $J=(z_0, 0)\subset (-1, 0)$
\[
\Vert \chi_q wG(\eta)\psi\Vert_{H^\sigma} \le C\big(\Vert \chi_qwH\Vert_{L^2(J, H^{\sigma+\mez})}+\Vert \chi_qw\partial_z H\Vert_{L^2(J, H^{\sigma-\mez})}\big).
\]
For the term $\Vert \chi_qw\partial_z H\Vert_{L^2(J, H^{\sigma-\mez})}$ we make use of the following identity (see $(3.21)$, \cite{ABZ2})
\[
\partial_z H=-\nabla\big((\partial_z\rho)\Lambda_2\widetilde\Phi\big)
\]
to have 
\[
\Vert \chi_qw\partial_z H\Vert_{L^2(J, H^{\sigma-\mez})}\le \Vert w(\partial_z\rho)\Lambda_2\widetilde\Phi\Vert_{L^2(J, H^{\sigma+\mez})_{ul}}.
\]
 On the other hand, we observe by definition \eqref{lambda} of $\Lambda_{1,2}$ that the terms in $H$ have the same structure as $(\partial_z\rho)\Lambda_2\widetilde\Phi$ and thus, it suffices to prove, for example, that
\[
\Vert w\nabla_x\rho\cdot\nabla_x\widetilde\Phi\Vert_{L^2(J, H^{\sigma+\mez})_{ul}}\le \mathcal{F}\bigl(\Vert   \eta  \Vert_{H^{s+ \mez}_{ul}}\bigr)\Vert w\psi \Vert_{H^{\sigma+1}_{ul}}.
\]
By virtue of Corollary \ref{coroetape1}, this reduces to  proving
\bq\label{boundDN}
\Vert w\nabla_x\rho\cdot\nabla_x\widetilde\Phi\Vert_{L^2(J, H^{\sigma+\mez})_{ul}}\le \mathcal{F}\bigl(\Vert   \eta  \Vert_{H^{s+ \mez}_{ul}}\bigr)\Vert w\nabla_{x,z}\widetilde\Phi\Vert_{X^\sigma_{ul}}.
\eq
We consider two cases:\\
{\it Case 1:} $-\mez\le \sigma\le s-1$. 
We apply the product rule \eqref{product:rule} with $s_0=\sigma+\mez,~s_1=s-\mez,~s_2=\sigma+\mez$ to obtain
\[
\Vert \chi_kw\nabla_x\rho\cdot\nabla_x\widetilde\Phi\Vert_{L^2(J, H^{\sigma+\mez})}\le C\Vert \widetilde\chi_k\nabla_x\rho\Vert_{L^\infty(J, H^{s-\mez})}\Vert \chi_kw\nabla_x\widetilde\Phi\Vert_{L^2(J, H^{\sigma+\mez})},
\]
from which \eqref{boundDN} follows in view of Lemma \ref{est-alpha}.\\
{\it Case 2:} $s-1\le \sigma\le s-\mez$. Since $\sigma+\mez>\frac{d}{2}$ we have the following well-known inequality
\[
\| ab\|_{H^{\sigma+\mez}}\les \|a\|_{L^\infty}\|b\|_{H^{\sigma+\mez}}+\|a\|_{H^{\sigma+\mez}}\|b\|_{L^\infty}.
\]
On the other hand,  $\sigma+\mez\le s$ and $\sigma\ge s-1>\frac{d}{2}$ so
\[
\| ab\|_{H^{\sigma+\mez}}\les \|a\|_{H^{s-\mez}}\|b\|_{H^{\sigma+\mez}}+\|a\|_{H^s}\|b\|_{H^\sigma}.
\]
Applying the preceding inequality to $a=\widetilde\chi_k\nabla_x\rho$ and $b=\chi_kw\nabla_x\widetilde\Phi$ yields
\begin{multline*}
\Vert \chi_kw\nabla_x\rho\cdot\nabla_x\widetilde\Phi\Vert_{L^2(J, H^{\sigma+\mez})}\le \|\widetilde\chi_k\nabla_x\rho\|_{L^\infty(J,H^{s-\mez})}\|\chi_kw\nabla_x\widetilde\Phi\|_{L^2(J, H^{\sigma+\mez})}\\
+\|\widetilde\chi_k\nabla_x\rho\|_{L^2(J, H^{s})}\|\chi_kw\nabla_x\widetilde\Phi\|_{L^\infty(J, H^{\sigma})}.
\end{multline*}
Lemma \ref{est-alpha} then implies the desired estimate \eqref{boundDN}. The proof is complete.
\end{proof}
\begin{rema}
Several comments are in order about Proposition \ref{coroetape2}. Let us recall a relating result on the exponential decay of the Dirichlet-Neumann operator (for $d=1$) obtained by Ming-Rousset-Tzvetkov \cite{MiRoTz}:
\begin{prop}[\protect{\cite[Proposition~3.2]{MiRoTz}}]\label{exp:MRT}
Assume that $\eta_*=-H$ and $\psi\in C^\infty_b(\xR)$ having an exponential decay:
\[
\exists \ld>0,~\forall j\in \xN,~j\ge 1, \exists C_j>0,~\forall x\in \xR,~|\partial^j_x\psi(x)|\le C_j e^{-\ld |x|}.
\]
Then for $\eta\in H^\infty(\xR)$ with $\eta-H\ge h>0$, $G(\eta)\psi$ also has an exponential decay, that is, there exist $0<\eps<\ld$ such that for any $j\in \xN$ we can find a constant $C'_j>0$ such that
\[
|\partial^j_xG(\eta)\psi)(x)|\le C'_je^{-\eps |x|}.
\]
\end{prop}
1. The advantage of Proposition \ref{exp:MRT} is that  it does not assume decay on $\psi$ itself but its derivatives, which is compatible with the solitons studied there. The authors were not interested in the way the estimates depend on the regularity of the surface $\eta$. \\ 
2.  Proposition \ref{exp:MRT}  is asymmetric in the sense that the exponential decay of $G(\eta)\psi$ is lower than the decay of $\psi$.\\
3.  Proposition \ref{coroetape2} assumes also that $\psi$ is decay (choosing for example $w(x)=e^{\varrho\langle x\rangle}$). However, compare to Proposition \ref{exp:MRT}, it has the following advantages:
\begin{enumerate}
\item[(i)]  Proposition \ref{coroetape2} holds in any dimension, with varying bottom ($\eta_*$ is only assumed to be in $C^0_b$) and allows domains with non smooth surfaces ($\eta\in C^{2+\eps}$).\\
\item [(ii)] The decay rate of $G(\eta)\psi$ is preserved, i.e., the same as the rate of $\psi$. Moreover, the exponential rate $\varrho$ of $\psi$ can be chosen as
\bq\label{explain:varrho}
\varrho\le \frac{1}{2C_*\Vert \eta-\eta_*\Vert^2_{L^\infty}}
\eq
(for some absolute constant $C_*=C_*(d)$) which is decreasing in the square of the fluid depth. Such a rate is therefore higher for shallow water but deteriorates when the depth tends to infinity. 
\end{enumerate}
\end{rema}
\subsection{Paralinearization of the Dirichlet--Neumann operator}
 We  denote by $\kappa$ the principal symbol of the Dirichlet-Neumann operator:
 $$
\kappa=  \big((1+ \vert \nabla_x \eta\vert^2 ) \vert \xi \vert^2 -(\nabla_x \eta \cdot \xi)^2\big)^\mez
$$
and define the remainder
 \begin{equation}\label{reste}
 R(\eta)\psi\ := G(\eta)\psi\ -T_{\kappa} \psi.
 \end{equation}
Our aim in this section is to prove the following weighted version of Theorem 3.11 in \cite{ABZ2}.
\begin{theo}\label{paralinDN}
Let  $s>1+{\frac{d}{2}}$ and $w\in \mathcal{W}_{po}(\varrho),~\varrho\ge 0$. Then there exists  $\mathcal{F}:\xR^+ \to \xR^+$ non decreasing such that for $0\leq t \leq s-\mez, $  $\eta  \in H^{s+\mez}_{ul}(\xR^d) $ satisfying~\eqref{condition} we have 
 $$\Vert wR(\eta)\psi\Vert_{H^{t }_{ul} } \leq \mathcal{F}\bigl(\Vert  \eta  \Vert_{  H^{s+ \mez}_{ul}} \bigr)
  \Vert w\psi\ \Vert_{H^{t+\mez}_{ul}} 
$$
provided that $w\psi \in H^{t+\mez}_{ul}(\xR^d)$.
\end{theo}
\begin{proof}
Let us fix a real number  $t\in [0, s-\mez]$.  By definition of the Dirichlet-Neumann operator, one has 
\begin{equation*}
G(\eta) \psi  = h_1 \partial_z\widetilde{\Phi}-h_2\cdot \nabla_x \widetilde{\Phi}  \big \arrowvert_{z=0}, \quad 
h_1  =  \frac{1 + \vert \nabla_x \rho\vert^2}{\partial_z \rho}, \quad h_2 = \nabla_x \rho. 
\end{equation*}
Let $A$  be the symbol of class $\Gamma^1_{\mez}(\xR^d\times J)$ given in Lemma 3.20, \cite{ABZ2}. We set
$$
\widetilde g_k= (\partial_z - T_A)(\chi_kw\widetilde\Phi),  \quad h_j\arrowvert_{z=0} = h_{j }^0, \quad j=1,2, \quad A\arrowvert_{z=0} = A_0.
$$ 
Then we can write
\begin{equation}
\begin{aligned}
\chi_k w G(\eta) \psi  &=  h_{1 }^0(\partial_z (\chi_kw\widetilde\Phi))\arrowvert_{z=0} - \chi_kwh_2^0\cdot \nabla_x \psi \\
&=h_1^0\widetilde{g}_k\arrowvert_{z=0}+h^0_1T_{A_0}(\chi_k\widetilde\chi_kw\psi) - \chi_k\widetilde\chi_kwh_2^0\cdot \nabla_x \psi  \\
&= h_1^0\widetilde{g}_k\arrowvert_{z=0} + h_1^0[T_{A_0}, \chi_k](\widetilde\chi_kw\psi) + \chi_k \big(h_1^0 T_{A_0} - h_2^0\cdot \nabla\big)(\widetilde\chi_kw\psi)\\
&\quad +\chi_kh_2^0\cdot \nabla(\widetilde\chi_kw\psi)-\chi_k\widetilde\chi_kwh_2^0\cdot \nabla_x \psi.
\end{aligned}
\end{equation}
Therefore, 
\[
\chi_k w G(\eta) \psi  =B_1+B_2,
\]
where 
\begin{gather*}
B_1= h_1^0\widetilde{g}_k\arrowvert_{z=0} + h_1^0[T_{A_0}, \chi_k](\widetilde\chi_kw\psi) + \chi_k \big(h_1^0 T_{A_0} - h_2^0\cdot \nabla\big)(\widetilde\chi_kw\psi), \\ B_2=\chi_kh_2^0\cdot (\nabla\widetilde\chi_k)w\psi+\chi_kh_2^0\cdot (\nabla w)\psi=\chi_kh_2^0\cdot (\nabla w)\psi.
\end{gather*}
The proof of Theorem 3.11, \cite{ABZ2} shows that
$$
 B_1=\chi_kT_{\kappa}(\widetilde\chi_kw\psi)+S
$$
with the remainder $S$ satisfies 
 $$\Vert S\Vert_{H^t} \leq  \mathcal{F}\bigl(\Vert  \eta  \Vert_{H^{s+ \mez}_{ul} })\Vert w\psi \Vert_{H^{t+\mez}_{ul} }.$$ 
Writing $\nabla w=wr_1$ with $r_1\in C^\infty_b(\xR^d)$. Since $h^0_2\in H^{s-\mez}(\xR^d)_{ul}$  with norm bounded by $\mathcal{F}(\Vert  \eta  \Vert_{H^{s+ \mez}_{ul} })$ and $t\le s-\mez$ the product rule yields
\begin{align*}
\|B_{2}\|_{H^t}&=\|\chi_kh_2^0\cdot r_1w\psi\|_{H^t}\\
&=\| \widetilde\chi_k h_2^0\cdot r_1\|_{H^{s-\mez}}\| \chi_kw\psi\|_{H^t}\\
&\le   \mathcal{F}\big(\Vert  \eta  \Vert_{H^{s+ \mez}_{ul} }\big)\Vert w\psi \Vert_{H^{t}_{ul}}.
\end{align*}
We have proved the following result
\[
\chi_k w G(\eta) \psi =\chi_kT_\kappa(\widetilde\chi_kw\psi)+\widetilde S,\quad \Vert \widetilde S\Vert_{H^t}\le \mathcal{F}\big(\Vert  \eta  \Vert_{H^{s+ \mez}_{ul} }\big)\Vert w\psi \Vert_{H^{t+\mez}_{ul}}.
\]
The proof of Theorem \ref{paralinDN} boils down to showing that the commutator $T:=\chi_k[T_\kappa,\widetilde\chi_kw]\psi$ satisfies
\bq\label{commutator:T}
 \Vert T\Vert_{H^t}\le \mathcal{F}\big(\Vert  \eta  \Vert_{H^{s+ \mez}_{ul} }\big)\Vert w\psi \Vert_{H^{t+\mez}_{ul}}.
\eq
 Indeed, introduce $\overline{\chi}\in C^\infty_c(\xR^d),~\overline{\chi}=1$ on the support of $\widetilde\chi$ and define 
also $\overline{\chi}_q(\cdot)=\overline\chi(\cdot-q)$ for all $q\in \xZ^d$. Denoting $w_k=\widetilde\chi_kw$ and noticing that $w_k=w_k\overline{\chi}_k$ we can write 
\begin{align*}
T&=\chi_kT_\kappa(w_k\overline\chi_k \psi)-\chi_kw_k\overline{\chi_k}T_\kappa \psi\\
&=\chi_kT_\kappa(w_k\overline\chi_k \psi)-\chi_kw_k[\overline\chi_k, T_\kappa]\psi-\chi_kw_kT_\kappa (\overline{\chi_k}\psi)\\
&=\chi_k[T_\kappa, w_k](\overline\chi_k\psi)-\chi_kw_k[\overline\chi_k, T_\kappa]\psi=:R_1-R_2.
\end{align*}
1. $R_1$ can be written as
\begin{align*}
R_1&=\chi_k[T_\kappa, T_{w_k}](\overline\chi_k\psi)+\chi_kT_\kappa[(w_k-T_{w_k})(\overline\chi_k\psi)]-\chi_k(w_k-T_{w_k})(T_\kappa(\overline\chi_k\psi))\\
&=:R_{1,a}+R_{1,b}-R_{1,c}.
\end{align*}
a) For $R_{1,a}$ we apply the symbolic calculus for paradifferential operators in Kato's spaces  (cf. Theorem $7.16$ $(ii)$, \cite{ABZ2}) to $\kappa\in \Gamma^1_1,~w_k\in \Gamma^0_1$ (see also Theorem \ref{calc:symb} $(ii)$ in the Appendix) to have
\[
\Vert R_{1,a}\Vert_{H^t}\le C \mathcal{F}\big(\Vert  \eta  \Vert_{H^{s+ \mez}_{ul} }\big)\Vert w_k\Vert_{W^{1,\infty}}\Vert \overline\chi_k\psi \Vert_{H^t_{ul}}.
\]
Now by properties $(i)$,~$(ii)$ in Definition \ref{classW}, the weight $w$ satisfies
\[
\Vert w_k\Vert_{W^{1,\infty}}\le Cw(k)\quad\text{and}\quad w(k)\Vert \overline\chi_k\psi \Vert_{H^t_{ul}}\le C\Vert w\psi \Vert_{H^t_{ul}}.
\]
Consequently, 
\[
\Vert R_{1,a}\Vert_{H^t}\le \mathcal{F}\big(\Vert  \eta  \Vert_{H^{s+ \mez}_{ul} }\big)\Vert w\psi \Vert_{H^t_{ul}}.
\]
b) For $R_{1,b}$ one first uses the boundedness of $T_\kappa$ from $H^{t+1}_{ul}$ to $H^t_{ul}$ (see Theorem $7.16$ $(i)$, \cite{ABZ2}) to have
\[
\Vert R_{1,b}\Vert_{H^t}\le \mathcal{F}\big(\Vert  \eta  \Vert_{H^{s+ \mez}_{ul} }\big)\Vert (w_k-T_{w_k})(\overline\chi_k\psi)\Vert_{H^{t+1}_{ul}}
\]
Next, the paraproduct rule in Proposition $7.18$, \cite{ABZ2} gives for $n$ large enough
\[
\Vert (w_k-T_{w_k})(\overline\chi_k\psi)\Vert_{H^{t+1}_{ul}}\le C\Vert w_k\Vert_{H^n_{ul}}\Vert \overline\chi_k\psi\Vert_{H^{t}_{ul}}
\]
As in a) we remark that $\Vert w_k\Vert_{H^n_{ul}}\le Cw(k)$ and hence obtain the desired estimate for $R_{1,b}$. The term $R_{1,c}$ can be handled using exactly the same method. In summary, we have proved an estimate better than needed:
\[
\Vert R_1\Vert_{H^t}\le \mathcal{F}\big(\Vert  \eta  \Vert_{H^{s+ \mez}_{ul} }\big)\Vert w\psi \Vert_{H^{t}_{ul}}.
\]
2. To study $R_2$ we decompose
\[
R_2=w_k[T_{\overline{\chi}_k},T_\kappa]\psi+w_k(\overline\chi_k-T_{\overline{\chi}_k})T_\kappa\psi-w_kT_\kappa[(\overline\chi_k-T_{\overline{\chi}_k})\psi]
\]
By the same arguments as in the the study of $R_1$ but using instead the symbolic calculus in Kato's spaces with weights in Theorem \ref{calc:symb} together with inequality \eqref{est:a-Ta} one ends up with 
\[
\Vert R_2\Vert_{H^t}\le \mathcal{F}\big(\Vert  \eta  \Vert_{H^{s+ \mez}_{ul} }\big)\Vert w\psi \Vert_{H^{t}_{ul}}.
\]
The proof of Proposition \ref{paralinDN} is complete.
\end{proof}
\begin{rema}
In the proof above, it is for the study of the remainder $R_2$ that we need to restrict the weight $w$ to the class of "polynomial weights" $\mathcal{W}_{po}(\varrho)$.
\end{rema}
\subsection{A weighted estimate for $\Phi$}
We use the elliptic regularity theorem \ref{regell} to prove a weighted estimate for $\Phi$--solution to \eqref{eqPhi}, which will be used later in proving a contraction estimate for Dirichlet-Neumann operator.
\begin{lemm}
Let   $s>1+{\frac{d}{2}}$ and $w\in \mathcal{W}_{po}(\varrho),~\varrho\ge 0$. With $\mu>0$ satisfying (\ref{delta1}) and $\Phi_q,~\psi_q$ as in section \ref{defiDN} there exists a non-decreasing function $\mathcal{F}$ independent of $q$ such that 
\[
\sum_{k\in \xZ^d}\| we^{\frac{\mu}{2} \langle x-q\rangle}\chi_k \nabla_{x,z}\widetilde{\Phi}_q\|_{L^{\infty}(J\times \xR^d)}\le \mathcal{F}(\|\eta\|_{H^{s+\frac{1}{2}}_{ul}}) \|w\psi_q\|_{H^s}.
\]
\label{l1estimate}
\end{lemm}
\begin{proof} We remark that $w\psi_q\in H^s(\xR^d)$ for every $q\in \xZ^d$ provided that $\psi\in H^s_{ul}(\xR^d)$. It is clear that
\[
\|we^{\frac{\mu}{2} \langle x-q\rangle}\chi_k  \nabla_{x,z}\widetilde{\Phi}_q\|_{L^{\infty}(J\times \xR^d)}\le  e^{-\frac{\mu}{4}\langle k-q\rangle} \|we^{3\mu/4 \langle x-q\rangle} \nabla_{x,z} \widetilde{\Phi}_q\|_{L^{\infty}(J\times \xR^d)}.     
  \]
Consider the weight $we^{3\mu/4\langle x-q\rangle}\in \mathcal{W}$ which has semi-norms independent of $q$. Applying Theorem \ref{regell} to $\Phi_q$ (with $\sigma=s-1$) and taking into account Remark \ref{remark: -mez} , we may estimate
\[
\begin{aligned}
&\sum_{k\in \xZ^d}\|we^{\frac{\mu}{2}\langle x-q\rangle}\chi_k  \nabla_{x,z}\widetilde{\Phi}_q\|_{L^{\infty}(J\times \xR^d)}\\
 &\le  C\sum_{k\in \xZ^d} e^{-\frac{\mu}{4}\langle k-q\rangle} \|we^{3\mu/4 \langle x-q\rangle} \nabla_{x,z}\widetilde{\Phi}_q\|_{L^{\infty}(J\times \xR^d)}\\
&\le C \|we^{3\mu/4\langle x-q\rangle}  \nabla_{x,z}\widetilde{\Phi}_q\|_{L^{\infty}(J, H^{s-1}(\xR^d))_{ul}}\\
&\le  \mathcal{F}(\|\eta\|_{H^{s+\frac{1}{2}}_{ul}})\left\{ \|we^{3\mu/4\langle x-q\rangle}\psi_q\|_{H^s_{ul}}+\|we^{3\mu/4\langle x-q\rangle}\nabla_{x,z}\widetilde{\Phi}_q\|_{X^{-\frac{1}{2}}_{ul}(J)}\right\}\\
&\le \mathcal{F}(\|\eta\|_{H^{s+\frac{1}{2}}_{ul}})\left\{ \|w\psi_q\|_{H^s}+\|we^{3\mu/4\langle x-q\rangle}\nabla_{x,z}\widetilde{\Phi}_q\|_{X^{-\frac{1}{2}}_{ul}(J)}\right\}.
\end{aligned}
\]
Remark that in the first inequality, we have used the trivial fact that $ \sum_{k\in \xZ^d} e^{-\frac{\mu}{4}\langle k-q\rangle}$ is finite and independent of $q$.\\
To complete the proof we need to prove that 
\bq
\|we^{3\mu/4 \langle x-q\rangle}\nabla_{x,z}\widetilde{\Phi}_q\|_{X^{-\frac{1}{2}}_{ul}(J)}\le \mathcal{F}(\|\eta\|_{H^{s+\frac{1}{2}}_{ul}}) \|w\psi_q\|_{H^{\frac{1}{2}}}.
\eq
However, using interpolation inequality as in step 2 of the proof of Proposition \ref{rec00}, it suffices to show that
\bq
\|we^{3\mu/4 \langle x-q\rangle}\nabla_{x,z}\widetilde{\Phi}_q\|_{L^2(J, L^2)_{ul}}\le \mathcal{F}(\|\eta\|_{H^{s+\frac{1}{2}}_{ul}})\|w\psi_q\|_{H^{\frac{1}{2}}}.
\label{phiqw1}
\eq
Indeed, by virtue of \eqref{phiqw} one can estimate
\begin{align*}
\|\chi_p we^{\frac{3\mu}{4}\langle x-q\rangle}\nabla_{x,z}\widetilde\Phi_q\|_{L^2(J, L^2)}&\lesssim  e^{-\frac{\mu}{4}\langle p-q\rangle}w(p)\|\chi_p e^{\mu\langle x-q\rangle}\nabla_{x,z}\widetilde\Phi_q\|_{L^2(J, L^2)}\\
&\lesssim  e^{-\frac{\mu}{4}\langle p-q\rangle}w(p)\mathcal{F}(\|\eta\|_{W^{1,\infty}})\|\chi_q \psi\|_{H^{\frac{1}{2}}}\\
&\lesssim  e^{-\frac{\mu}{4}\langle p-q\rangle}w(p)w(q)^{-1}\mathcal{F}(\|\eta\|_{W^{1,\infty}})\|w\chi_q \psi\|_{H^{\frac{1}{2}}}\\
&\lesssim  e^{-\frac{\mu}{4}\langle p-q\rangle}\langle p-q\rangle^{\varrho}(\|\eta\|_{W^{1,\infty}})\|w\chi_q \psi\|_{H^{\frac{1}{2}}}\\
&\lesssim \mathcal{F}(\|\eta\|_{W^{1,\infty}})\|w\chi_q \psi\|_{H^{\frac{1}{2}}},
\end{align*}
which is the desired bound.
\end{proof}
\begin{rema}
As in Proposition \ref{rec00}, Lemma \ref{paralinDN} can be formulated for weights in the class $\mathcal{W}_{ex}(\varrho)$ with $\varrho$ sufficiently small. 
\end{rema}
\section{Weighted contraction for the Dirichlet-Neumann operator}\label{contrac:DN}\label{contractDN}
The main ingredient in proving the contraction for the Dirichlet-Neumann operator is  the contraction estimate for solutions to the elliptic problem \eqref{eqPhi}. The key idea then is to compare the two variational solutions $\Phi_j$ after changing the variables $\Phi_j$ to $\widetilde\Phi_j$ as in \eqref{change}. However, after straightening the fluid domains by the diffeomorphism \eqref{diffeo}, the new domains will depend on their upper surface. To overcome this, we use a slightly different diffeomorphism as follows.\\
Given $\eta_*\in C^0_b(\xR^d)$ and $h>0$, there exists $\widetilde \eta\in C_b^\infty(\xR^d)$ such that
\bq\label{tildeeta}
\eta_*(x)<\widetilde\eta(x)<\eta_*(x)+\frac{h}{3},\quad\forall x\in \xR^d.
\eq
Then, because $\eta_j>\eta_*+h$ we set
\[
\begin{cases}
\Omega_{1, j}&=\{(x, y): x\in \xR^d,~ \eta_j(x)-\frac{h}{3}<y<\eta_j(x) \},\\
\Omega_{2, j}&=\{(x, y): x\in \xR^d,~ \widetilde\eta(x)\le y\le\eta_j(x)-\frac{h}{3} \},\\
\Omega_{3, j}&=\{(x, y): x\in \xR^d,~ \eta_*(x)<y<\widetilde\eta(x)\},\\
\Omega_j&=\Omega_{1,j}\cup\Omega_{2,j}\cup\Omega_{3,j},
\end{cases}
\]
and 
\[
\begin{cases}
\widetilde\Omega_{1}&=\xR^d_x\times (-1, 0)_z,\\
\widetilde\Omega_{2}&=\xR^d_x\times [-2, -1]_z,\\
\widetilde\Omega_{3}&=\{ (x, z)\in\xR^d\times (-\infty -2): z+2+\widetilde \eta(x)>\eta_*(x)\},\\
\widetilde\Omega&=\widetilde\Omega_{1}\cup\widetilde\Omega_{2}\cup\widetilde\Omega_{3}.
\end{cases}
\]
Remark that $\widetilde \Omega$ depends on $\eta_*,~h$ but not on $\eta_j$. Thus, we can define
\bq\label{diffeoj}
\rho_j(x, z)=
\begin{cases}
&\rho_{1,j}(x,z)=(1+z)e^{\delta z\langle D_x \rangle }\eta_j(x) -z\left [e^{-(1+ z)\delta\langle D_x \rangle }\eta_j(x) -\frac{h}{3}\right],\quad \text{in} ~ \widetilde{\Omega}_{1},\\
&\rho_{2,j}(x,z)= (2+z)\left[e^{\delta (z+1)\langle D_x \rangle }\eta_j(x) -\frac{h}{3}\right]-(1+z)\widetilde\eta, \quad \text{in }~ \widetilde{\Omega}_{2},\\
& \rho_{3,j}(x,z)=z+2+\widetilde\eta(x), \quad \text{in}~ \widetilde{\Omega}_{3}.
\end{cases}
\eq
\begin{lemm}\label{bound:dz}
Assume that $\eta_j\in W^{1,\infty}(\xR^d)$, $j=1, 2$. There exists an absolute constant $C>0$ such that if 
\[
C\delta \lA \eta_j\rA_{W^{1,\infty}(\xR^d)}\le h,~j=1,2
\]
then the mappings $(x, z)\mapsto (x, \rho_j(x, z))$ are Lipschitz diffeomorphisms from $\widetilde\Omega$ to  $\Omega_j$ and there exists a constant $c_0>0$ such that $\partial_z\rho_j\ge c_0$ a.e. in $\Omega$.
\end{lemm}
\begin{proof} Observe first that $\rho_{k, j}$ are Lipschitz for $k=1, 2, 3;~j=1, 2$. Clearly,  $(x, z)\mapsto (x, \rho_{3,j}(x, z))$ are  diffeomorphisms from $\widetilde\Omega_3$ to  $\Omega_{3,j}$ and $\partial_z\rho_{3,j}=1\ge c_0>0$. The same properties hold for $\rho_{1,j}$ as in \eqref{diffeo}. We now prove it for $\rho_{2,j}$. Notice first that \[
\rho_{2,j}(-1, x)=\eta_j-\frac{h}{3},\quad \rho_{2,j}(-2, x)=\widetilde \eta.
\]
Compute now
\begin{align*}
\partial_z \rho_{2,j}&=e^{\delta (z+1)\langle D_x \rangle }\eta_j(x) -\frac{h}{3}-(2+z)\delta e^{\delta (z+1)\langle D_x \rangle}\langle D_x \rangle\eta_j-\widetilde\eta \\
&=e^{\delta (z+1)\langle D_x \rangle }\eta_j(x) -\eta_j(x)-(2+z)\delta e^{\delta (z+1)\langle D_x \rangle}\langle D_x \rangle\eta_j+\eta_j(x)-\widetilde\eta -\frac{h}{3}.
\end{align*}
By writing $e^{\delta (z+1)\langle D_x \rangle }\eta_j-\eta_j=\delta (z+1)\int_0^1e^{r\delta (z+1)\langle D_x \rangle }\langle D_x \rangle \eta_jdr $ we deduce that
\begin{align*}
&\lA e^{\delta (z+1)\langle D_x \rangle }\eta_j -\eta_j\rA _{L^\infty(\xR^d)} +\lA (2+z)\delta e^{\delta (z+1)\langle D_x \rangle}\langle D_x \rangle\eta_j\rA_{L^\infty(\xR^d)}\\
&\quad \le C\delta \lA \eta_j\rA_{W^{1,\infty}(\xR^d)}\le \frac{h}{6}
\end{align*}
for $\delta>0$ small enough. On the other hand, thanks to \eqref{condition} and \eqref{tildeeta} it holds that
\[
\eta_j-\widetilde\eta-\frac{h}{3}=(\eta_j-\eta_*)+(\eta_*-\widetilde \eta)-\frac{h}{3}>h-\frac{h}{3}-\frac{h}{3}=\frac{h}{3}
\]
and thus $\partial_z\rho_{2,j}\ge \frac{h}{3}-\frac{h}{6}=\frac{h}{6}$ in $\widetilde \Omega_2$. Therefore, we can conclude that  $(x, z)\mapsto (x, \rho_{2,j}(x, z))$ are diffeomorphisms from $\widetilde\Omega_2$ to  $\Omega_{2,j}$. 
\end{proof}
With the functions $\rho_j$ above we  denote for every $f:\Omega\to \xR$ 
\bq\label{imagej}
\widetilde f_j(x,z)=f(x, \rho_j(x, z))
\eq
and as in \eqref{lambda} we define the differential operators $\Lambda^j=(\Lambda^j_1, \Lambda^j_2)$. Hereafter, we denote  $J=(-2, 0)$ and assume that 
\bq\label{condition:eta}
\eta_j\in H^{s+\mez}_{ul}(\xR^d),~s>1+\frac{d}{2},~j=1, 2.
\eq
\begin{lemm}\label{lem:wp} Let $w\in \mathcal{W}_{po}(\varrho),~\varrho\ge 0$. We have  $\Lambda^1-\Lambda^2=\wp\partial_z=(\wp_1, \wp_2)\partial_z$ with $\wp=0$ for $z<-2$ and
\bq\label{est:wp}
\|w\wp\|_{L^2(J, L^2(\xR^d))_{ul}}\le \mathcal{F}(\|(\eta_1, \eta_2)\|_{H^{s+\mez}_{ul}\times H^{s+\mez}_{ul}})\|w(\eta_1-\eta_2)\|_{H^{\mez}_{ul}}.
\eq
\end{lemm}
\begin{proof}
By definition, one gets
\begin{align*}
\wp_1&=\frac{\partial_z(\rho_2-\rho_1)}{\partial_z\rho_1\partial_z\rho_2},\\
\wp_2&=-\frac{\nabla_x(\rho_2-\rho_1)}{\partial_z\rho_1}-\nabla_x\rho_2\frac{\partial_z(\rho_2-\rho_1)}{\partial_z\rho_1\partial_z\rho_2}
\end{align*}
so in $\widetilde \Omega_3$,  $\wp=0$. To obtain \eqref{est:wp} one writes
\[
\|w\wp\|_{L^2(J, L^2(\xR^d))_{ul}}\le \|w\wp\|_{L^2((-1,0), L^2(\xR^d))_{ul}}+\|w\wp\|_{L^2((-2,-1), L^2(\xR^d))_{ul}}
\]
to use definition \eqref{diffeoj}, the fact that $\widetilde \eta\in C^\infty_b(\xR^d)$ and  the $\mez$-smoothing effect of the Poisson kernel, which is Proposition \ref{L2smooth} applied with $r=1$.
\end{proof}
\begin{theo}\label{diffPhi:low}
Let $\psi_j\in H_{ul}^\mez(\xR^d)$ and $\Phi_j, j=1,2$ be the unique solution in $H^1_{ul}(\Omega_j)$ of the problem
\bq\label{eq:Phij}
 \Delta_{x,y}\Phi_j = 0 \text{ in } \Omega, \quad  \Phi_j\arrowvert_{\Sigma} = \psi_j, \quad \frac{\partial \Phi_j}{\partial \nu}\arrowvert_ \Gamma = 0.
 \eq
Set  $\eta=\eta_1-\eta_2,~\psi=\psi_1-\psi_2,~\widetilde{\Phi}=\widetilde{\Phi}_1-\widetilde{\Phi}_2$  where $\widetilde\Phi_j$ is the image of $\Phi_j$ as in \eqref{imagej}. Then for every $w\in \mathcal{W}_{po}(\varrho),~\varrho\ge 0$ there exists a nonnegative function $\cF$ such that 
\bq
\|w\nabla_{x,z}\widetilde{\Phi}\|_{X^{-\frac{1}{2}}_{ul}(J)}\le \mathcal{F}(\|(\eta_1, \eta_2)\|_{H^{s+\frac{1}{2}}_{ul}\times H^{s+\frac{1}{2}}_{ul}})\left(\|w\eta\|_{H^{s-\mez}_{ul}}\|\psi_{2}\|_{H^s_{ul}}+\|w\psi\|_{H^{\mez}_{ul}}\right).
\label{dphi1/2}
\eq
\end{theo}
 For the proof of this result, we shall apply Lemma \ref{l1estimate} for $\widetilde\Phi_q$. However, $\widetilde\Phi_q$ here is the image of $\Phi_q$ via the diffeomorphism corresponding to one of $\rho_j$ defined  by \eqref{diffeoj} instead of \eqref{diffeo}. We want the same result as Lemma \ref{l1estimate} in this situation. To have this, we notice that on $J=(-2, 0)$, $\rho_j$ is comprised of two functions $\rho_{1,j}$ for $z\in (-1,0)$ and $\rho_{2,j}$ for $z\in (-2, -1]$. The function $\rho_{1,j}$ possesses the same properties as $\rho$ does and so does $\rho_{2,j}$ since $\widetilde \eta\in C_b^\infty\subset H^\infty_{ul}$. Therefore, we obtain 
\begin{lemm}\label{l1estimate'}
Let $w\in \mathcal{W}_{po}(\varrho),~\varrho\ge 0$ and $\Phi_{j,q},~\psi_{j,q}$,~$j=1,2,~q\in \xZ$ as in section \ref{defiDN}. There exists $\mathcal{F}_1$ non-decreasing such that: if $0<\mu \mathcal{F}_1(\lA \eta_j\rA_{H^s_{ul}})\le 1$ then one can find a non-decreasing function $\mathcal{F}$ independent of $q$ such that 
\[
\sum_{k\in \xZ^d}\| we^{\frac{\mu}{2} \langle x-q\rangle}\chi_k  \nabla_{x,z}\widetilde{\Phi}_{j,q}\|_{L^{\infty}(J\times \xR^d)}\le \mathcal{F}(\|\eta_j\|_{H^{s+\frac{1}{2}}_{ul}}) \|w\psi_{j,q}\|_{H^s}.
\]
\label{l1estimate'}
\end{lemm}
\begin{proof} ({\it of Theorem \ref{diffPhi:low}}) For simplicity in notations we shall denote $\mathcal{F}=\mathcal{F}(\|\eta_1\|_{H^{s+\frac{1}{2}}_{ul}}, \|\eta_2\|_{H^{s+\frac{1}{2}}_{ul}})$ which may change from line to line. We proceed in the following steps.\\
{\it Step 1.} Let $\Phi_{j,q}=u_{j,q}+\underline{\psi}_{j,q}$ where $u_{j,q}$ is the variational solution characterized by (\ref{eqvar}). After changing the variables,  \eqref{eqvar} becomes
\[
\int_{\widetilde\Omega}\Lambda^j\widetilde{\Phi}_{j,q}\Lambda^j\theta J_jdX=0, \quad\forall \theta\in H^{1,0}(\widetilde\Omega),~  j=1,2
\]
with the Jacobian $J_j=|\partial_z\rho_j|=\partial_z\rho_j\ge c_0>0$ a.e. in $\widetilde \Omega$ (by Lemma \ref{bound:dz}). \\
\hk Set $\widetilde \Phi_q=\widetilde\Phi_{1,q}-\widetilde \Phi_{2, q},~{\underline\psi}_q=\underline\psi_{1,q}-\underline\psi_{2,q}$ and choose 
$$\theta=e^{2\delta g_{\ep}}(\widetilde{\Phi}_q-\widetilde{\underline\psi}_q)\in H^{1,0}(\widetilde\Omega)$$
 where $g_{\ep}=\frac{\langle x-q\rangle }{1+\ep\langle x-q\rangle}$.
It follows that 
\[
\left |\int_{\widetilde\Omega}\Lambda^1\widetilde \Phi_q \Lambda^1\theta J_1dX\right |\le \sum_{j=1}^3A_j,
\]
\[
\begin{cases}
A_1=\int_{\widetilde\Omega}|(\Lambda^1-\Lambda^2)\widetilde \Phi_{2,q}\Lambda^1\theta J_1|dX,\\
A_2=\int_{\widetilde\Omega}|\Lambda^2\widetilde \Phi_{2,q}(\Lambda^1-\Lambda^2)\theta J_1|dX,\\
A_3=\int_{\widetilde\Omega}|\Lambda^2\widetilde \Phi_{2,q}\Lambda^2\theta (J_1-J_2)|dX.
\end{cases}
\]
By Lemma \eqref{lem:wp} we know that $\Lambda^1-\Lambda^2=0$ in $\widetilde\Omega_3$. Likewise, $J_1-J_2=\partial_z\rho_1-\partial_z\rho_2=0$ in $\widetilde\Omega_3$. Consequently, with $\widetilde \Omega_0=\xR^d\times J$ we have $A_j,~j=1,2,3$ are equal to the corresponding integrals over $\widetilde\Omega_0$.\\
{\it Step 2.} (Estimate for $A_1$) First of all, we remark that
\bq\label{Lambda^i}
\Lambda^j(e^{2\delta g_{\ep}}U)=e^{2\delta g_{\ep}}\Lambda^jU+(0, U)2\delta e^{2\delta g_{\ep}}\nabla g_{\ep}.
\eq
Using Lemma \ref{lem:wp} and formula \eqref{Lambda^i} with $j=1,~U=\widetilde{\Phi}_q-\widetilde{\underline\psi}_q$ one can write
\begin{align*}
A_1&=\int_{\widetilde\Omega_0}e^{2\delta g_{\ep}}|\wp\partial_z\widetilde \Phi_{2,q}\Lambda^1(\widetilde\Phi_q-\widetilde{\underline\psi}_q) J_1|dX+2\delta \int_{\widetilde\Omega_0}e^{2\delta g_{\ep}}|\nabla g_{\ep}\wp_2\partial_z\widetilde \Phi_{2,q}(\widetilde\Phi_q-\widetilde{\underline\psi}_q) J_1|dX\\
&:=A_{1,1}+A_{1,2}.
\end{align*}
Since $\lA J_j\rA_{L^\infty_{x,z}}\le \cF$, we may estimate 
\bqa
A_{1,1}&\le & \cF\int_{\widetilde\Omega_0}e^{2\delta g_{\ep}}|\wp\partial_z\widetilde \Phi_{2,q}\Lambda^1(\widetilde\Phi_q-\widetilde{\underline\psi}_q)|dX\\
&\le &\cF\|\wp e^{\delta g_{\ep}}\partial_z\widetilde \Phi_{2,q}\|_{L^2(\widetilde\Omega_0)}\|e^{\delta g_{\ep}}\Lambda^1(\widetilde\Phi_q-\widetilde{\underline\psi}_q)\|_{L^2(\widetilde\Omega_0)}.
\eqa
On the other hand, there holds
\bq\label{est:product}
\begin{aligned}
\|f_1f_2\|_{L^2(J, L^2(\xR^d))}&\le \sum_k \|\widetilde\chi_kf_1\chi_kf_2\|_{L^2(J, L^2(\xR^d))}\\
&\le  \|f_1\|_{L^2(J, L^2(\xR^d)_{ul})}\sum_k\|\chi_kf_2\|_{L^{\infty}(J\times \xR^d)}.
\end{aligned}
\eq
Now we choose $\delta>0$ such that 
\bq
\delta \mathcal{F}_1(\|\eta_2\|_{H^{s+\frac{1}{2}}_{ul}})\le  \frac{1}{2}
\label{delta3}
\eq
then the condition of Lemma \ref{l1estimate'} is fulfilled with $\mu=2\delta$. It then follows from \eqref{est:product} and Lemma \ref{lem:wp} that
\bq\label{L2key}
\begin{aligned}
\|\wp e^{\delta g_{\ep}}\partial_z\widetilde \Phi_{2,q}\|_{L^2(\widetilde\Omega_0)}&\le  \|w\wp w^{-1}e^{\delta \langle x-q\rangle}\partial_z\widetilde \Phi_{2,q}\|_{L^2(J, L^2(\xR^d))}\\
&\le  \|w\wp\|_{L^2(J, L^2(\xR^d)_{ul})}\sum_k\|\chi_kw^{-1}e^{\delta \langle x-q\rangle}\partial_z\widetilde \Phi_{2,q}\|_{L^{\infty}(J\times \xR^d)}\\
&\le  \mathcal{F} \|w\eta\|_{H^{\mez}_{ul}}\|w^{-1}\psi_{2,q}\|_{H^s}.
\end{aligned}
\eq
Therefore, 
\[
A_{1,1}\le \cF\|e^{\delta g_{\ep}}\Lambda^1(\widetilde\Phi_q-\widetilde{\underline\psi}_q)\|_{L^2(\widetilde\Omega_0)} \|w\eta\|_{H^{\mez}_{ul}}\|w^{-1}\psi_{2,q}\|_{H^s}.
\]
For $A_{1,2}$ we have
\[
 A_{1,2} \le 2\delta\cF \| e^{\delta g_{\ep}}\wp_2\partial_z\widetilde \Phi_{2,q}  \|_{L^2(\widetilde\Omega_0)}\|e^{\delta g_{\ep}}(\widetilde\Phi_q-\widetilde{\underline\psi}_q)\|_{L^2(\widetilde\Omega_0)}.
\]
The first $L^2$-norm on the right-hand side is already estimated by \eqref{L2key}. For the second term on the right-hand side, one applies the Poincar\'e inequality in Lemma \ref{lemm:Poincare} (with $\mathcal{O}=\Omega_{1,1}\cup \Omega_{2, 1}$, which is diffeomorphic to $\widetilde \Omega_0$) to the image of $(\widetilde\Phi_q-\widetilde{\underline\psi}_q)$ under the inverse of $\rho_1(x,z),~(x,z)\in\widetilde\Omega_0$ and  then changes the variables back to $\widetilde\Omega_0$ to derive
\bq\label{est:Poincare}
\begin{aligned}
\|e^{\delta g_{\ep}}(\widetilde\Phi_q-\widetilde{\underline\psi}_q)\|_{L^2(\widetilde\Omega_0)}\le \cF\|e^{\delta g_{\ep}}\partial_z(\widetilde\Phi_q-\widetilde{\underline\psi}_q)\|_{L^2(\widetilde\Omega_0)}\le \cF\|e^{\delta g_{\ep}}\Lambda^1(\widetilde\Phi_q-\widetilde{\underline\psi}_q)\|_{L^2(\widetilde\Omega_0)}
\end{aligned}
\eq
from which we deduce that $A_{1,2}$ satisfies the same estimate as $A_{1,1}$ does and hence, so does $A_1$, i.e., 
\bq\label{est:A1}
A_1\le \cF\|e^{\delta g_{\ep}}\Lambda^1(\widetilde\Phi_q-\widetilde{\underline\psi}_q)\|_{L^2(\widetilde\Omega_0)} \|w\eta\|_{H^{\mez}_{ul}}\|w^{-1}\psi_{2,q}\|_{H^s}.
\eq
{\it Step 3.} (Estimates for $A_2,~A_3$) By Lemma \ref{est:wp} we have
\[
(\Lambda^1-\Lambda^2)\theta =\wp e^{2\delta g_\eps}\partial_z (\widetilde\Phi_q-\widetilde{\underline\psi}_q).
\]
It follows that 
\[
A_2\le \cF \|\wp e^{\delta g_{\ep}}\Lambda^2\widetilde \Phi_{2,q}\|_{L^2(\widetilde\Omega_0)}\|e^{\delta g_{\ep}}\partial_z(\widetilde\Phi_q-\widetilde{\underline\psi}_q)\|_{L^2(\widetilde\Omega_0)}.
\]
Using the definition of $\Lambda^2$ and the same method as in \eqref{L2key} one obtains that the first term is also bounded by the right-hand side of \eqref{L2key}. On the other hand, it is easy to see the second term is bounded by  $\cF\|e^{\delta g_{\ep}}\Lambda^1(\widetilde\Phi_q-\widetilde{\underline\psi}_q)\|_{L^2(\widetilde\Omega_0)}$. Therefore, $A_2$ also satisfies the bound \eqref{est:A1}.\\
For $A_3$ one uses the formula \eqref{Lambda^i} to get $A_3\le A_{3,1}+A_{3,2}$ with 
\begin{align*}
A_{3,1}&=\int_{\widetilde\Omega_0}e^{2\delta g_{\ep}}|\Lambda^2\widetilde \Phi_{2,q}\Lambda^2(\widetilde\Phi_q-\widetilde{\underline\psi}_q) (J_1-J_2)|dX,\\
A_{3,2}&=\delta \int_{\widetilde\Omega_0}e^{2\delta g_{\ep}}|\nabla g_{\ep}\Lambda^2_2\widetilde \Phi_{2,q}(\widetilde\Phi_q-\widetilde{\underline\psi}_q)(J_1-J_2)|dX.
\end{align*}
First, $A_{3,2}$ is estimated by 
$\|(J_1-J_2) e^{\delta g_{\ep}}\Lambda^2\widetilde \Phi_{2,q}\|_{L^2(\widetilde\Omega_0)}\|e^{\delta g_{\ep}}(\widetilde\Phi_q-\widetilde{\underline\psi}_q)\|_{L^2(\widetilde\Omega_0)}.
$
The second term is estimated by \eqref{est:Poincare} and the first term is estimated as in \eqref{L2key} with $\wp$ replaced by $J_1-J_2$ which satisfies
$
\| w(J_1-J_2)\|_{L^2(J, L^2_{ul})}\le \cF\|w\eta\|_{H^\mez_{ul}}.
$
Similarly, \[
A_{3,1}\le \|(J_1-J_2) e^{\delta g_{\ep}}\Lambda^2\widetilde \Phi_{2,q}\|_{L^2(\widetilde\Omega_0)}\|e^{\delta g_{\ep}}\Lambda^2(\widetilde\Phi_q-\widetilde{\underline\psi}_q)\|_{L^2(\widetilde\Omega_0)}.
\] 
We only need to study the second term on the right-hand side. With $u\defn\widetilde\Phi_q-\widetilde{\underline\psi}_q$ one has
$
\Lambda^2_1u=\frac{\partial_z\rho_1}{ \partial_z\rho_2}\Lambda^1_1u
$
which implies $
\|e^{\delta g_{\ep}}\Lambda^2_1u\|_{L^2(\widetilde\Omega_0)}\le \cF \|e^{\delta g_{\ep}}\Lambda^1_1u\|_{L^2(\widetilde\Omega_0)}.$ On the other hand, 
\[
\Lambda^2_2u=\nabla_xu-\frac{\nabla_x\rho_2}{\partial_z\rho_2}\partial_zu=\Lambda^1_2u+\left(\frac{\nabla_x\rho_1}{\partial_z\rho_1}-\frac{\nabla_x\rho_2}{\partial_z\rho_2}\right)\partial_z\rho_1\left(\frac{1}{\partial_z\rho_1}\partial_zu\right).
\]
Hence, $\|e^{\delta g_{\ep}}\Lambda^2_2u\|_{L^2(\widetilde\Omega_0)}\le \cF \|e^{\delta g_{\ep}}\Lambda^1u\|_{L^2(\widetilde\Omega_0)}$ and 
$\|e^{\delta g_{\ep}}\Lambda^2u\|_{L^2(\widetilde{\Omega}_0)}\le \cF\|e^{\delta g_{\ep}}\Lambda^1u\|_{L^2(\widetilde{\Omega}_0)}$.\\
In conclusion, we have proved that: for any (small) $\delta>0$ satisfying (\ref{delta3}), there holds
\begin{align}\label{diffphil1}
\left |\int_{\widetilde\Omega}\Lambda^1\widetilde \Phi_q \Lambda^1\theta J_1dX\right |
\le \cF\|e^{\delta g_{\ep}}\Lambda^1(\widetilde\Phi_q-\widetilde{\underline\psi}_q)\|_{L^2(\widetilde\Omega_0)} \|w\eta\|_{H^{\mez}_{ul}}\|w^{-1}\psi_{2,q}\|_{H^s}.
\end{align}
{\it Step 4.}  Next, in view of \eqref{Lambda^i} we write
\begin{equation}\label{diffphil2}
\begin{aligned}
&\int_{\widetilde\Omega}\Lambda^1\widetilde \Phi_q \Lambda^1\theta J_1dX\\
&=\int_{\widetilde\Omega}e^{2\delta g_{\ep}}\Lambda^1\widetilde \Phi_q\Lambda^1(\widetilde \Phi_q-\widetilde {\underline\psi}_q)J_1dX+2\delta \int_{\widetilde\Omega}\Lambda^1_2\widetilde \Phi_q.(\widetilde \Phi_q-\widetilde {\underline\psi}_q)e^{2\delta g_{\ep}}\nabla g_{\ep} J_1dX\\
&=\int_{\widetilde\Omega}e^{2\delta g_{\ep}}|\Lambda^1(\widetilde \Phi_q-\widetilde {\underline\psi}_q)|^2J_1dX+\int_{\widetilde\Omega}e^{2\delta g_{\ep}}\Lambda^1\widetilde {\underline\psi}_q\Lambda^1(\widetilde \Phi_q-\widetilde {\underline\psi}_q)J_1dX\\
&\quad+2\delta \int_{\widetilde\Omega}\Lambda^1_2(\widetilde \Phi_q-\widetilde {\underline\psi}_q)(\widetilde \Phi_q-\widetilde {\underline\psi}_q)e^{2\delta g_{\ep}}\nabla g_{\ep}J_1dX\\
&\quad+2\delta \int_{\widetilde\Omega}\Lambda^1_2\widetilde {\underline\psi}_q(\widetilde \Phi_q-\widetilde {\underline\psi}_q)e^{2\delta g_{\ep}}\nabla g_{\ep}J_1 dX:=B_1+B_2+B_3+B_4.
\end{aligned}
\end{equation}
Owing to the Poincar\'e inequality in Lemma \ref{Poincare} (applied with $\mathcal{O}=\Omega_1$) and a change of variables one has
\bq
|B_3|\le
 \delta \mathcal{F}_2(\|\eta_1\|_{H^{s+\frac{1}{2}}_{ul}})\|e^{\delta g_{\ep}}\Lambda^1(\widetilde \Phi_q-\widetilde {\underline\psi}_q)\|^2_{L^2(\widetilde\Omega)}
\eq
where $\mathcal{F}_2:\xR^+\to \xR^+$ is a non decreasing function. Likewise,
\bq\label{diffphil3}
|B_4|\le \delta \mathcal{F}_2(\|\eta_1\|_{H^{s+\frac{1}{2}}_{ul}})\|e^{\delta g_{\ep}}\Lambda^1\widetilde {\underline\psi}_q\|_{L^2(\widetilde\Omega)}\|\Lambda^1(\widetilde \Phi_q-\widetilde {\underline\psi}_q)e^{\delta g_{\ep}}\|_{L^2(\widetilde\Omega)}.
\eq
Finally, it is clear that
\bq\label{diffphil4}
|B_2|\le \mathcal{F}_2(\|\eta_1\|_{H^{s+\frac{1}{2}}_{ul}})\|e^{\delta g_{\ep}}\Lambda^1\widetilde {\underline\psi}_q\|_{L^2(\widetilde\Omega)}\|\Lambda^1(\widetilde \Phi_q-\widetilde {\underline\psi}_q)e^{\delta g_{\ep}}\|_{L^2(\widetilde\Omega)}.
\eq
Now, remark that there exists a constant $c_0$ depending only on $h$ such that $|J_1|\ge c_0$. Choose $\delta>0$ satisfying 
\bq
\delta \left\{\mathcal{F}_1(\|\eta_2\|_{H^{s+\frac{1}{2}}_{ul}})+ \mathcal{F}_2(\|\eta_1\|_{H^{s+\frac{1}{2}}_{ul}})\right\}= \min (\frac{c_0}{2}, \mez).
\label{delta2}
\eq
A combination of (\ref{diffphil1})-(\ref{diffphil4}) yields
\begin{align*}
\lA e^{\delta g_{\ep}}\Lambda^1(\widetilde \Phi_q-\widetilde {\underline\psi}_q)\rA_{L^2(\widetilde \Omega)}\le\cF \left\{\|w\eta\|_{H^{\mez}_{ul}}\|w^{-1}\psi_{2,q}\|_{H^s}+\|e^{\delta g_{\ep}}\Lambda^1\widetilde {\underline\psi}_q\|_{L^2(\widetilde\Omega)}\right\}.
\end{align*}
{\it Step 5.} Now, letting $\ep\to 0$ and taking into account properties $(i),~(ii)$ of $\underline{\psi}_q$ in section \ref{defiDN} lead to
\bq
\begin{aligned}
\|e^{\delta \langle x-q\rangle}\Lambda^1\widetilde \Phi_q\|_{L^2(\widetilde \Omega)}&\le \cF \left\{\|w\eta\|_{H^1_{ul}}\|w^{-1}\psi_{2,q}\|_{H^s}+\|e^{\delta \langle x-q\rangle}\Lambda^1\widetilde {\underline\psi}_q\|_{L^2(\widetilde \Omega)}\right\}\\
&\le \cF\left\{ \|w\eta\|_{H^1_{ul}}\|w^{-1}\psi_{2,q}\|_{H^s}+\|\psi_q\|_{H^{\mez}}\right\}.
\end{aligned}
\eq
Hence
\[
\|e^{\delta \langle x-q\rangle}\nabla_{x,z}\widetilde \Phi_q\|_{L^2(\widetilde \Omega)}\le \mathcal{F}\left\{\|w\eta\|_{H^1_{ul}}\|w^{-1}\psi_{2,q}\|_{H^s}+\|\psi_q\|_{H^{\mez}}\right\}.
\]
Consequently,
\bq
\begin{aligned}
\|\chi_pw\nabla_{x,z}\widetilde \Phi_q\|_{L^2(J, L^2(\xR^d))}&\le e^{-\delta\langle p-q\rangle}w(p)\mathcal{F}\left\{\|w\eta\|_{H^{\mez}_{ul}}\|w^{-1}\psi_{2,q}\|_{H^s}+\|\psi_q\|_{H^{\mez}(\Omega)}\right\}\\
&\le  e^{-\delta\langle p-q\rangle}w(p)w(q)^{-1}\mathcal{F}\left\{\|w\eta\|_{H^{\mez}_{ul}}\|{\psi}_{2,q}\|_{H^s}+\|w\psi_q\|_{H^{\mez}}\right\}.
\end{aligned}
\eq
Finally, we get
\bq\label{ctlow1}
\begin{aligned}
\|\chi_pw\nabla_{x,z}\widetilde \Phi\|_{L^2(J, L^2(\xR^d))}&\le  \sum_q\|\chi_pw\nabla_{x,z}\widetilde \Phi_q\|_{L^2(J, L^2(\xR^d))}\\
&\le \mathcal{F}\left\{ \|w\eta\|_{H^{\mez}_{ul}}\|{\psi}_2\|_{H^s_{ul}}+\|w\psi\|_{H^{\mez}_{ul}}\right\}.
\end{aligned}
\eq
{\it Step 6.} It remains to prove that $\|\chi_pw\nabla_{x,z}\widetilde \Phi\|_{L^{\infty}(J, H^{-\mez}(\xR^d))}$ is bounded by the right hand side of (\ref{dphi1/2}).\\
\hk The estimate of $\|\chi_pw\nabla_{x}\widetilde \Phi\|_{L^{\infty}(J, H^{-\mez}(\xR^d))}$ follows from (\ref{ctlow1}) and the interpolation Lemma \ref{interpo}. By the same lemma,  for $\|\chi_pw\partial_z\widetilde \Phi\|_{L^{\infty}(J, H^{-\mez}(\xR^d))}$ it remains to estimate $$\|\chi_pw\partial^2_z\widetilde \Phi\|_{L^2(J, H^{-1}(\xR^d))}.$$ 
For this purpose we use  equation (\ref{e:Phi}) below, satisfied by $\widetilde\Phi$ to have
\begin{equation}\label{diffPhi1}
\begin{aligned}
\|\chi_pw\partial^2_z\widetilde \Phi\|_{L^2(J, H^{-1}(\xR^d))}\le \, &\|\chi_pw\alpha_1\Delta\widetilde \Phi\|_{L^2(J, H^{-1}(\xR^d))}+\|\chi_pw\beta_1.\nabla\partial_z\widetilde \Phi\|_{L^2(J, H^{-1}(\xR^d))}\\
&+\|\chi_pw\gamma_1\partial_z\widetilde \Phi\|_{L^2(J, H^{-1}(\xR^d))}+\|\chi_pwF\|_{L^2(J, H^{-1}(\xR^d))}.
\end{aligned}
\end{equation}
In the above inequality, $\alpha_1,~\beta_1,~\gamma_1$ are defined as in \eqref{alpha}, for the diffeomorphism $\rho$ defined in term of $\eta_1$.\\
Because $-1<s-2$, the estimate \eqref{contractDN1} applied with $f=\psi_2$ implies the desired estimate for $\|\chi_pwF\|_{L^2(J, H^{-1}(\xR^d))}$. Concerning the other terms, the product rule \eqref{product:rule} gives
\begin{equation}\label{F:H^-1}
 \begin{aligned}
 \Vert \chi_pw\alpha_1\Delta\widetilde \Phi\Vert_{L^2(J, H^{-1})} &\leq C\Vert \widetilde\chi_p\alpha_1\Vert _{L^\infty(J, H^{s-\mez})} \Vert \chi_pw\Delta\widetilde \Phi \Vert_{L^2(J, H^{-1})}\\
 \Vert \chi_p w\beta_1 \cdot\partialx\partial_z \widetilde \Phi\Vert_{L^2(J, H^{-1})} &\leq C\Vert \widetilde\chi_p \beta_1 \Vert _{L^\infty(J, H^{s-\mez})} \Vert \chi_pw\partialx\partial_z \widetilde \Phi\Vert_{L^2(J, H^{-1})}\\
\Vert  \chi_pw\gamma_1\partial_z\widetilde \Phi\Vert_{L^2(J, H^{-1})} &\leq C\Vert \widetilde\chi_p\gamma_1 \Vert _{L^\infty(J, H^{s-\frac{3}{2}})} \Vert \chi_pw\partial_z \widetilde \Phi \Vert_{L^2(J, L^2)}.
  \end{aligned}
\end{equation}
Owing to \eqref{ctlow1} we are left with the estimates for the first term on the right-hand side of the above inequalities. Again, this is done along the same line as in the proof of Lemma \ref{est-alpha} noticing that $\widetilde\eta\in C_b^\infty\subset H^\infty_{ul}$. This completes the proof.
\end{proof}
We are now in position to derive the weighted estimate for the Dirichlet-Neumann operator:
\begin{theo}\label{lip:G}
Assume that $s>1+\frac{d}{2}$. Then for every $w\in \mathcal{W}_{po}(\varrho),~\varrho\ge 0$ there exists $\mathcal{F}:\xR^+\to \xR^+$ non decreasing such that for all $\eta_1, \eta_2\in H^{s+\frac{1}{2}}_{ul}(\xR^d)$ and $f\in H^s_{ul}(\xR^d)$
we have
\[
\|w[G(\eta_1)-G(\eta_2)]f\|_{H^{s-\frac{3}{2}}_{ul}}\le \mathcal{F}(\|(\eta_1, \eta_2)\|_{H^{s+\frac{1}{2}}_{ul}\times H^{s+\frac{1}{2}}_{ul}})\|w(\eta_1-\eta_2)\|_{H^{s-\frac{1}{2}}_{ul}}\|f\|_{H^{s}_{ul}}.
\]
\end{theo}
\begin{proof}
Let $\Phi_j$ defined as in \eqref{eq:Phij} (with $\psi_j=f$, $j=1,2$) and $\widetilde\Phi_j$ be its image via the diffeomorphism $(x,z)\mapsto (x, \rho_j(x,z))$ given by  \eqref{diffeoj}. We have from definition \eqref{G=} of the Dirichlet-Neumann operator
\bq\label{DNformula}
G(\eta_j)f=\left(\frac{1+|\nabla_x\rho_j|^2}{\partial_z\rho_j}\partial_z\widetilde{\Phi}_j-\nabla_x\rho_j\nabla_x\widetilde{\Phi}_j \right)\Big\arrowvert_{z=0}.
\eq
Recall from \eqref{equ:modifie} that each $\widetilde\Phi_j$ satisfies the elliptic problem
\[
(\partial_z^2 + \alpha_j \Delta_x + \beta_j\cdot \nabla_x \partial_z - \gamma_j \partial_z) \widetilde\Phi_j = 0,
\]
where $\alpha_j,~\beta_j,~\gamma_j$ are defined as in \eqref{alpha} with $\eta$ replaced by $\eta_j$.\\
 Set $\widetilde{\Phi}=\widetilde{\Phi}_1-\widetilde{\Phi}_2$ then
\bq\label{e:Phi}
\left(\partial_z^2+\alpha_1\Delta_x+\beta_1\cdot\nabla\partial_z-\gamma_1\partial_z\right)\widetilde{\Phi}=F
\eq
with
\[
F=\left\{(\alpha_2-\alpha_1)\Delta_x+(\beta_2-\beta_1)\cdot\nabla\partial_z-(\gamma_2-\gamma_1)\partial_z \right\}\widetilde{\Phi}_2.
\]
We fix $z_0\in (-1, 0)$ and set $I_0=(z_0, 0)$,~$I=(-1, 0)$. We first prove that
\bq\label{contractDN0}
\|w\nabla_{x,z}\widetilde \Phi\|_{X^{s-\mezl}_{ul}(I_0)}\le \mathcal{F}(\|(\eta_1, \eta_2)\|_{H^{s+\frac{1}{2}}_{ul}\times H^{s+\frac{1}{2}}_{ul}})\|w(\eta_1-\eta_2)\|_{H^{s-\frac{1}{2}}_{ul}}\|f\|_{H^{s}_{ul}}.
\eq
To prove the preceding estimate, we claim that
\bq
\|wF\|_{L^2(I, H^{s-2})_{ul}}\le \mathcal{F}(\|(\eta_1, \eta_2)\|_{H^{s+\frac{1}{2}}_{ul}\times H^{s+\frac{1}{2}}_{ul}})\|w(\eta_1-\eta_2)\|_{H^{s-\frac{1}{2}}_{ul}}\|f\|_{H^{s}_{ul}}.
\label{contractDN1}
\eq
Indeed, the $H^s_{ul}$ version of the product rule \eqref{product:rule} (see Proposition $7.3$ $(i)$, \cite{ABZ2}) applied with  $s_0=s-2,~s_1=s-1,~s_2=s-2$ yields
\bq\label{claim:F1}
\begin{aligned}
\|wF\|_{L^2(I, H^{s-2})_{ul}}&\les  \Vert w(\alpha_2-\alpha_1)\Vert_{L^2(I, H^{s-1})_{ul}}\Vert \Delta_x\widetilde{\Phi}_2\Vert_{L^\infty(I, H^{s-2})_{ul}}\\
&\quad +\Vert w(\beta_2-\beta_1)\Vert_{L^2(I, H^{s-1})_{ul}}\Vert \nabla\partial_z\widetilde{\Phi}_2\Vert_{L^\infty(I, H^{s-2})_{ul}}\\
&\quad +\Vert w(\gamma_2-\gamma_1)\Vert_{L^2(I, H^{s-2})_{ul}}\Vert \partial_z\widetilde{\Phi}_2\Vert_{L^\infty(I, H^{s-1})_{ul}}.
\end{aligned}
\eq
On the other hand, applying Proposition \ref{L2smooth} for the $\mez$-smoothing effect of the Poisson kernel in weighted spaces gives 
\begin{multline}\label{claim:F2}
\Vert w(\alpha_2-\alpha_1)\Vert_{L^2(I, H^{s-1})_{ul}}+\Vert w(\beta_2-\beta_1)\Vert_{L^2(I, H^{s-1})_{ul}}+\Vert w(\gamma_2-\gamma_1)\Vert_{L^2(I, H^{s-2})_{ul}}\\ \le
\mathcal{F}(\|(\eta_1, \eta_2)\|_{H^{s+\frac{1}{2}}_{ul}\times H^{s+\frac{1}{2}}_{ul}})\|w(\eta_1-\eta_2)\|_{H^{s-\frac{1}{2}}_{ul}}.
\end{multline}
Remark that in Corollary \ref{coroetape1}  one can replace the assumption $z_0\in (-1, 0)$ by $z_0\in J=(-2, 0)$ because on $(-2, 1]$ the diffeomorphism $\rho_j$ satisfies the same bounds as the diffeomorphism defined in \eqref{diffeo} does (again, this is true because $\widetilde\eta\in C^\infty_b(\xR^d)$). This remark applied with $\sigma=s-1$ and $w\equiv 1$ leads to 
\bq\label{claim:F3}
\Vert \nabla_{x,z}\widetilde\Phi_2\Vert_{L^\infty(I, H^{s-1}_{ul})}\le \mathcal{F}(\|\eta_2\|_{H^{s+\frac{1}{2}}_{ul}})\|f\|_{H^s_{ul}}.
\eq
Putting together \eqref{claim:F1}, \eqref{claim:F2}, \eqref{claim:F3} one obtains the claim \eqref{contractDN1}.\\
Since $\widetilde\Phi|_{z=0}=0$, with the aid of Theorem \ref{regell} (which is applicable since $\rho_{1,j}$ and $\rho$ in \eqref{diffeo} have exactly the same form),  the proof of (\ref{contractDN0}) now reduces to estimate $\|w\nabla_{x,z}\widetilde \Phi\|_{X^{-\mez}_{ul}(I)}$. This is a consequence of Theorem \ref{diffPhi:low} applied with $\psi_1=\psi_2=f$ (and the fact that $I\subset J$). \\
In view of \eqref{DNformula}, to obtain the bound for $w[G(\eta_1)f-G(\eta_2)]f$ in $H^{s-\tdm}_{ul}$ it is necessary to bound  $\|w\nabla_{x,z}\widetilde\Phi\|_{H^{s-\frac{3}{2}}_{ul}}$ at $z=0$. More precisely, we shall prove that 
\[
\|w\nabla_{x,z}\widetilde\Phi\arrowvert_{z=0}\|_{H^{s-\tdm}_{ul}}\le \mathcal{F}(\|(\eta_1, \eta_2)\|_{H^{s+\frac{1}{2}}_{ul}\times H^{s+\frac{1}{2}}_{ul}})\|w(\eta_1-\eta_2)\|_{H^{s-\frac{1}{2}}_{ul}}\|f\|_{H^s_{ul}}.
\]
To this end,  we use the  argument in step $6$  of the proof of Theorem \ref{diffPhi:low}. By virtue of Lemma \ref{interpo} and \eqref{contractDN0}, 
 we then only need to estimate $\|w\partial_z^2\widetilde\Phi\|_{L^2(I_0, H^{s-2})_{ul}}$, which in turn follows by using equation (\ref{e:Phi}) together with the estimate \eqref{contractDN1} and the product rule \eqref{product:w}. Finally, using (\ref{DNformula}) and the product rule \eqref{product:w} once again, we conclude the proof of Theorem \ref{lip:G}.
\end{proof}
\begin{rema}
Theorem \ref{diffPhi:low} is also a crucial ingredient in proving contraction of the remainder $R$ appearing in the reformulation of water waves system-equation $(4.4)$ in Proposition $4.2$, \cite{ABZ2}. Notice that our estimate \eqref{dphi1/2} is sufficient for this purpose because 
\[
\lA w(\psi_1-\psi_2)\rA_{H^{\mez}_{ul}}\le \lA w(\psi_1-\psi_2)\rA_{H^{s-1}_{ul}}
\]
owing to the fact that $s>1+\frac{d}{2}$.
\end{rema}
\section{Proof of the main results}
\subsection{Proof of Theorem \ref{main}}
 The contraction estimate in Theorem \ref{main} was proved  in \cite{ABZ1} (see Theorem $5.1$) for classical Sobolev spaces and then in \cite{ABZ2} for Kato's spaces. Both use the following scheme:
\begin{enumerate}[1)]
\item  study the Dirichlet-Neumann operator: bound estimates and paralinearization 
\item contraction estimate for the Dirichlet-Neumann operator 
\item paralinearization of the difference equations (after reformulation)
\item estimates for the good unknown
\item back to the original unknowns. 
 \end{enumerate}
\hk Here, we shall follow the same scheme as above. The first two items are the real new points in our problem and have been studied in Section \ref{descriptDN} and \ref{contractDN}. For the last three items we need a para-differential machinery in Kato's spaces with weights and this is established  in Appendix \ref{paramachinery}. The key point in this machinery is that: whenever we estimate $S(u, v)$ in weighted norms, where $S$ is an operator of two variables, we are always able to shift the weight to $u$ or $v$. Having this in hand, items $3), 4), 5)$ follow line by line those in \cite{ABZ2} and \cite{ABZ1}: one only need to replace $\Vert\cdot \Vert_{H^\sigma_{ul}}$ or $\Vert\cdot \Vert_{H^\sigma}$ by $\Vert w\cdot\Vert_{H^\sigma_{ul}}$ in the relevant estimates ($w$ is the weight). We conclude the proof.
\subsection{Proof of Corollary \ref{coro:main}}We need to show how \eqref{eq.lipschitz} implies \eqref{eq.lipschitz1}. To this end, it suffices to prove that there exist $0<T_1\le T$ and $N>0$ (both are independent of $U_j$) such that
\bq\label{est:Mj}
\lA U_j\rA_{L^{\infty}([0, T_1], \mathcal{H}_{ul}^s)}\le N, ~j=1,2.
\eq
Define the Sobolev norms of the solutions as 
\[
M^j_\sigma(\tau)=\lA U_j\rA_{L^{\infty}([0, \tau], \mathcal{H}_{ul}^\sigma)}, ~ \forall \tau\in [0, T],~\forall j\ge 0.
\]
Let us recall the a priori estimate derived in \cite{ABZ2}: for any $1+{\frac{d}{2}}<\sigma\le s$ and $\mathcal{T}>0$ one can find a non decreasing function $\cF:\xR^+\to \xR^+$ such that
\bq\label{apriori}
M^j_s(\mathcal{T})\le \cF\big(M^j_{\sigma}(0) +\mathcal{T}M^j_{\sigma}(\mathcal{T})\big)\big(M^j_s(0) +\mathcal{T}M^j_s(\mathcal{T}) \big),\quad\forall j\ge 0.
\eq
Fix $s_0\in \left(1+\frac{d}{2}, s\right)$. Since each $U_j$ is a solution to the gravity waters system in $C^0([0, T], \mathcal{H}_{ul}^{s_0})$, the estimate \eqref{apriori} gives for some non decreasing $\mathcal{F}_1:\xR^+\to \xR^+$ (independent of $U$)
\[
M^j_{s_0}(\tau)\le \mathcal{F}_1(M^j_{s_0}(0) +\tau M^j_{s_0}(\tau)),~ \forall \tau\in [0, T],~\forall j\ge 0.
\]
According to Theorem \ref{theo:ABZ}, $U_j$ is continuous in time with value in $\mathcal{H}_{ul}^{s_0}$ since $s_0<s$. Consequently, $M^j_{s_0}(\tau)$ is continuous in $\tau$. In addition, $M^j_{s_0}(0)$ can be bounded by some constant independent of $j$, say $A$. The standard argument then gives the existence of $T_0\in (0, T]$ and $N>0$, both are independent of $U_j$ (but depend on $A$), such that 
\bq\label{Ms_0}
M^j_{s_0}(\tau)\le N,~\forall \tau\in [0, T_0],\quad\forall j\ge 0.
\eq
Applying again the estimate \eqref{apriori} with $\sigma=s_0<s$ we get for some non-decreasing function $\mathcal{F}:\xR^+\to \xR^+$ (independent of $U$)
\[
M^j_{s}(\tau)\le \mathcal{F}(M^j_{s_0}(0) +T_0M^j_{s_0}(\tau))\left(M^j_s(0)+\tau M^j_s(\tau) \right), ~\forall \tau\in [0, T_0],~\forall j\ge 0.
\]
By \eqref{Ms_0}, this implies
\[
M^j_{s}(\tau)\le \mathcal{F}(N(1+T_0))\left(M^j_s(0)+\tau M^j_s(\tau) \right),~ \forall \tau\in [0, T_0].
\]
Now, let $T_1\in (0, T_0]$ satisfying
\[
T_1\mathcal{F}(N(1+T_0))\le \mez 
\]
one deduces
\[
M^j_s(T_1)\le 2\mathcal{F}(N(1+T_0))M^j_s(0),~\forall j\ge 0
\]
which concludes the proof.
 \appendix 
\section{Paradifferential calculus in Kato's spaces with weights}
\label{paramachinery}
In this section, we adapt the paradifferential machinery for the presence of weights which can be of independent interest. The proofs of these results follow those in \cite{ABZ2} but we need to take some care (so we only present the proof whenever it is necessary). We recall first various spaces which will be used in the sequel.
\begin{defi}\label{XY}
Let $ p\in [1,+\infty],$ $J =(z_0,0), z_0<0$ and $\sigma \in \xR$. \\
1. The space $L^p(J, H^\sigma(\xR^d))_{ul}$ is defined as the space of measurable functions 
$u$ from $J_z\times \xR^d_x $ to $ \xC$ such that
$$
\Vert u \Vert_{L^p(J, H^\sigma(\xR^d))_{ul}}:= \sup_{q\in \xZ^d}
\Vert \chi_q u \Vert_{L^p(J, H^\sigma(\xR^d))} <+\infty.
$$  
2. We set
  \begin{equation*}\label{X,Y}
\begin{aligned}
  X^\sigma_{ul}(J) &= L^\infty(J, H^\sigma(\xR^d))_{ul}\cap L^2(J, H^{\sigma+\mez}(\xR^d))_{ul}\\
  Y^\sigma_{ul} (J) & = L^1(J, H^\sigma(\xR^d))_{ul} +  L^2(J, H^{\sigma - \mez}(\xR^d))_{ul}
  \end{aligned}
\end{equation*}
endowed with their natural norms.\\
The same spaces without subscript "ul" are defined for classical Sobolev spaces.
\end{defi}
Notice that $L^\infty(J, H^\sigma(\xR^d))_{ul} = L^\infty(J, H^\sigma_{ul}(\xR^d))$.
\begin{nota} 
 For $t\in \xR$, we denote $\lceil t\rceil$ the smallest integer strictly greater than or equal $t$. 
\end{nota}
 \subsection{Weighted continuity of pseudo-differential operators}
In \cite{ABZ2}, the authors proved the continuity of pseudo-differential operators on the framework of $L^2$ based uniformly local Sobolev spaces. Here, we perform similar results with the presence of weights in classes $\mathcal{W}_{po}(\varrho),~\varrho>0$ (see Definition \ref{classW}), which are composed of functions that are at most polynomial growth. For the sake of clarity, let us redefine this class.
\begin{defi}\label{defi:Wp}
For every $\varrho \ge 0$, we define $\mathcal{W}_{po}(\varrho)$ to be the class of all functions $w:\xR^d\to (0, \infty)$ satisfying the following conditions:
\begin{enumerate}
\item[(i)]  $r_1:=\frac{\nabla w^{-1}}{w^{-1}}$ and $r'_1:=\frac{\nabla w}{w}$ belong to $C^\infty_b(\xR^d)$, where $w^{-1}(x)=1/w(x)$,\\
\item [(ii)] for any $C_1>0$, there exists $C_2>0$ such that for any $x_0\in \xR^d$, there hold 
\[
w(x)\le C_2w(x_0)\quad\text{and}\quad w(x)^{-1}\le C_2w(x_0)^{-1}\quad\forall x\in \xR^d,~|x-x_0|\le C_1,
\]
\item[(iii)] there exists $C>0$ such that for any $x,~y\in \xR^d$ we have $w(x)w^{-1}(y)\le C\langle x-y\rangle^\varrho$.\\
\end{enumerate}
\end{defi}
\begin{rema}\label{rema:Wpo}
1. For all $\ld\in \xR$ and $C>1$, the functions $\langle x\rangle^\ld,~\ln(C+|x|^2)$ belong to $\mathcal{W}_{po}(|\ld|)$ and $\mathcal{W}_{po}(1)$ respectively. For every $t\in \xR\setminus\{0\}$, the function $e^{t\langle x\rangle}$ does not belong to any class $\mathcal{W}_{pol}(\varrho)$.\\
2. If $w_j\in \mathcal{W}_{po}(\varrho_j),~j=1,2$ then $w_1w_2\in \mathcal{W}_{po}(\varrho_1+\varrho_2)$.\\
3. If $w\in \mathcal{W}_{po}(\varrho)$ and $w>0$ then $w^\ld\in \mathcal{W}_{po}(\varrho|\ld|)$ for any $\ld\in \xR$.\\
4. Condition $(iii)$ in Definition \ref{defi:Wp} is equivalent to \\
$(iii')$ there exist $M_1,~M_2>0$ and $\ld_1,~\ld_2\in \xR$ such that
\[M_1\langle x\rangle^{\ld_1}\le w(x)\le M_2\langle x\rangle^{\ld_2},\quad\forall x\in \xR^d.
\]
Although $(iii')$ seems to be more natural than $(iii)$, in the following proofs, condition $(iii)$ is more convenient. 
\end{rema}
 We denote by $S^m_{1,0}$ the standard H\"ormander's class  of symbols $p\in C^\infty(\xR^d \times \xR^d)$ satisfying 
$$  \vert D_\xi^\alpha D_x^\beta p(x,\xi)\vert \leq C_{\alpha,\beta}(1+\vert \xi \vert)^{m-\vert \alpha \vert}\quad \forall \alpha, \beta \in \xN^d, \forall (x,\xi) \in \xR^d \times \xR^d .$$
\begin{prop}\label{pseudo}
Let $P$ be a pseudo-differential operator whose symbol $p$ belongs to $S^{m}_{1,0}$ and let $w$ be a weight in $\mathcal{W}_{po}(\varrho),~\varrho\ge 0$. Then for any $s\in \xR$, there exists $C>0$ such that 
\[\lA wPu\rA_{H^s_{ul}}\le C\lA wu\rA_{H^{s+m}_{ul}},
\]
provided that the right hand side is finite.
\label{pseudo1}
\end{prop}
\begin{proof}
We write\bq
w\chi_kPu=\sum_{|k-q|\le 2}w\chi_kP\chi_qu+\sum_{|k-q|>3}w\chi_kP\chi_qu:=A+\sum_{|k-q|>3}B_{k,q}.
\eq
Since $\chi_qu=(\chi_qw(x)u)(\widetilde\chi_qw(x)^{-1})\in H^{s+m}(\xR^d)$, we have by properties $(i),~(ii),~(iii)$ of the weight $w$, the product rule \eqref{product:rule} and the classical pseudo-differential calculus that
\bqa
\|w\chi_kP\chi_qu\|_{H^{s}}&\le & Cw(k)\|P\chi_qu\|_{H^{s}}\\
&\le &  Cw(k)\|\chi_qu\|_{H^{s+m}}\\
&\le & Cw(k)w(q)^{-1}\|w\chi_qu\|_{H^{s+m}}\\
&\le &C \langle k-q \rangle^{\varrho}\|w\chi_qu\|_{H^{s+m}}\\
&\le & C \|wu\|_{H^{s+m}_{ul}},
\eqa
 provided $|k-q|\le 2$. Thus, 
\bq\label{pseudo:near}
A\le C\|wu\|_{H^{s+m}_{ul}}.
\eq
To bound  the second part, we fix $n_0\in \xN$, $n_0\ge s$. We shall prove 
\bq\label{DxB}
\|D^{\alpha}_xB_{k,q}\|_{L^2(\xR^d)}\le \frac{C}{\langle k-q\rangle^{d+1}}\|w u\|_{H^{s+m}_{ul}}, \quad |\alpha|\le n_0
\eq
which implies the desired estimate for $\sum_{|k-q|>3}B_{k,q}$.\\
By the presence of $\chi_k$, $\|D^{\alpha}_xB_{k,q}\|_{L^2(\xR^d)}\le C\|D^{\alpha}_xB_{k,q}\|_{L^{\infty}(\xR^d)}$. We have 
\[
D^{\alpha}_xB_{k,q}(x)=\langle D^{\alpha}_xK(x,\cdot), \chi_qu\rangle
\]
with 
\bq
K(x, y)=(2\pi)^{-d}\int_{\xR^d}e^{i(x-y)\xi}p(x, \xi)d\xi\chi_k(x)w(x)\widetilde{\chi}_q(y).
\label{K}
\eq
Fix $n_1\in \xN, n_1\ge -(s+m)$ and $\beta\in \xN^d, |\beta|\le n_1$. Let $\gamma\in \xN^d$ be such that $|\gamma|=N$ with 
\bq\label{choose:N:pseudo}
N\ge \max (m+n_0+n_1+ d+1, \varrho+d+1).
\eq
 Multiplying $D_x^{\alpha}D_y^{\beta}K(x,y)$  by $(x-y)^{\gamma}$ and integrating by parts with a remark that $|x-y|\ge \delta |k-q|$ (for some $\delta>0$) on the support of $\chi_k(x)w(x)\widetilde{\chi}_q(y)$, we obtain
\[
\la D_x^{\alpha}D_y^{\beta}K(x,y)\ra\le \frac{C_{\beta, d, \ld}}{\langle k-q \rangle^{N}}w(k)\sum_{|\beta_1|\le |\beta|}\la \partial^{\beta_1}\widetilde{\chi}_q(y)\ra.
\]
It follows that 
\bq
\begin{aligned}
|D^{\alpha}_xB_{k,q}(x)|&\le  \|D^{\alpha}_xK(x,\cdot)\|_{H^{-(s+m)}}\|\chi_qu\|_{H^{s+m}}\\
&\le  \frac{C}{\langle k-q\rangle^{N}}w(k)\|\chi_qu\|_{H^{s+m}}\\
&\le  \frac{C}{\langle k-q \rangle^{N}}w(k)w(q)^{-1}\|\chi_qwu\|_{H^{s+m}}\\
&\le  \frac{C}{\langle k-q \rangle^{N}}\langle k-q \rangle^{\varrho}\|w u\|_{H^{s+m}_{ul}}\\
&\le  \frac{C}{\langle k-q \rangle^{d+1}}\|w u\|_{H^{s+m}_{ul}}
\end{aligned}
\eq
 which proves \eqref{DxB}.
\end{proof}
In a particular case the proof above gives the  following more precise result.
\begin{prop}\label{pseudoh}
Let $m \in \xR$, and $w\in \mathcal{W}_{po}(\varrho),~\varrho\ge0$. Let $h(\xi)= \widetilde{h}\big (\frac{\xi}{\vert \xi \vert}\big)\vert \xi \vert^m \psi(\xi)$ with $\widetilde{h} \in C^\infty(\xS^{d-1})$, and 
\bq\label{cutoff:psi}
 \psi \in C ^\infty(\xR^d),\quad \psi(\xi) = 1~\text{if}~\vert \xi \vert \geq 1,\quad \psi(\xi) = 0~\text{if}~\vert \xi \vert \leq \mez. 
\eq
Then for every $s \in \xR$ and 
\bq\label{choose:r}
r>\lceil m\rceil+\lceil s\rceil+\lceil -(m+s)\rceil+\lceil \varrho\rceil+\frac{3d}{2}+1,
\eq
there exists a constant $C$ such that
$$\Vert  wh(D_x) u \Vert_{H^s_{ul}(\xR^d)} \leq C \Vert \widetilde{h} \Vert_{H^{r}(\xS^{d-1})} \Vert  wu \Vert_{H^{s+m}_{ul}(\xR^d)},$$
provided that the right-hand side is finite.
\end{prop}
Remark that the condition on $r$ above comes from the choice of $N$ in \eqref{choose:N:pseudo}, plus ${\frac{d}{2}}+\eps$ derivatives from Sobolev embeddings. Next,  tracking the proof of Lemma 7.10 in \cite{ABZ2} and Proposition \ref{pseudo} above, we easily obtain the following proposition. 
 \begin{prop}\label{L2smooth}
Let $r>0$,  and  $m\in\xR$ and $w\in \mathcal{W}_{po}(\varrho),~\varrho\ge0$ . Let $p\in S^r_{1,0}(\xR^d), a\in S^m_{1,0}(\xR^d)$ be two symbols with constant coefficients. We assume that there exists $c_0>0$ such that for all $\xi \in \xR^d$ we have $p(\xi)  \geq c_0 \vert \xi \vert^r$. Then for all $s\in \xR$ and $I = [0,T],$ one can find a positive constant $C$ such that
\begin{equation}\label{smoothing}
 \Vert we^{-tp(D)}a(D)u\Vert_{L^\infty(I, H^s)_{ul}} + \Vert w e^{-tp(D)}a(D)u\Vert_{L^2(I, H^{s+\frac{r}{2}})_{ul}} \leq C \Vert wu \Vert_{H^{s+m}_{ul}},
 \end{equation}
provided that the right-hand side is finite.
\end{prop}
\subsection{Para-differential calculus with weights}

Assuming the theory of para-differential calculus for classical Sobolev spaces  (see \cite{MePise}) and for uniformly local Sobolev spaces (see \cite{ABZ2}), we present in this section such a theory with the presence of weights.\\
Given $m\in \xR,~\rho \ge 0$ we denote  by $\Gamma_\rho^m(\xR^d)$ the class of symbols of order $m$ and by $T_a$ the associated para-differential operator as in Definition $7.15$, \cite{ABZ2}. In particular, $\dot{\Gamma}^m_\rho(\xR^d)$ denotes the subspace of $\Gamma^m_\rho(\xR^d)$ which consists of symbols $a(x,\xi)$  homogeneous of degree $m$ with respect to $\xi$.\\
To deal with the weights in the class $\mathcal{W}_{po}(\varrho)$, for any symbol $a\in \Gamma^m_{\rho}$ and any real number $s$, let us define the semi-norm
\begin{equation}\label{seminorm}
M^m_\rho(a, s)_{\varrho}= \sup_{\vert \alpha \vert \leq I(d, m, s)_\varrho}\sup_{\vert \xi \vert \geq \mez} \Vert (1+ \vert \xi \vert)^{ \vert \alpha \vert -m}\partial_\xi^\alpha a (\cdot,\xi)\Vert_{W^{\rho,\infty}(\xR^d)},
\end{equation}
where $I(d, m, s)_{\varrho}$ is the smallest even integer strictly greater than
\bq
\lceil m\rceil+\lceil s\rceil+\lceil -(m+s)\rceil+\lceil \varrho\rceil+\frac{5d}{2}+2.
\label{I}
\eq
If $a$ is a symbol independent of $\xi$,  the associated operator $T_a$ is called a paraproduct and we have the formal decomposition of Bony 
$$au=T_au+T_ua+R(a,u).$$
\subsubsection{Symbolic calculus}
The following technical lemmas will be used in proving results on symbolic calculus.
 \begin{lemm}\label{ulN}
Let $\mu \in \xR$, $w\in \mathcal{W}_{po}(\varrho),~\varrho\ge 0$ and $N \geq \varrho+d+1$. Then there exists $C>0$ such that
\begin{equation}\label{N>d+1}
\sup_{x\in \xR^d} \Vert w(x)\langle x-\cdot \rangle^{-N} u \Vert_{H^\mu(\xR^d)} \leq  C \Vert wu \Vert_{H_{ul}^\mu(\xR^d)}
\end{equation}
provided that the right hand side is finite.
\end{lemm}
\begin{proof}
We write
  $$
w(x)\langle x-y\rangle^{-N} \chi_q (y)u (y)  = w(x)w(y)^{-1} \frac{1}{\langle x-q \rangle^N} \frac{\langle x-q \rangle^N}{ \langle x-y\rangle^{ N}}\widetilde{\chi}_q(y)w(y) \chi_q (y)u (y).$$
  Since the function $y \mapsto \frac{\langle x-q \rangle^N}{ \langle x-y\rangle^{ N}}\widetilde{\chi}_q(y)$ belongs to $W^{\infty,\infty}(\xR^d)$ with semi-norms uniformly bounded (independently of $x$ and $q$), we deduce that 
 \begin{align*}
  \Vert w(x)\langle x-\cdot \rangle^{-N} u \Vert_{H^\mu(\xR^d)}&\le \sum_{q \in \xZ^d}w(x)\Vert \langle x-\cdot \rangle^{-N} \chi_qu \Vert_{H^\mu} \leq C_N \sum_{q\in \xZ^d} \frac{w(x)w(q)^{-1}}{\langle x-q \rangle^N} \Vert   w u \Vert_{H_{ul}^\mu} \\
& \leq C_N \sum_{q\in \xZ^d} \frac{\langle  x-q\rangle ^{\rho}}{\langle x-q \rangle^N} \Vert    wu \Vert_{H_{ul}^\mu} \leq C'_N \Vert  wu\Vert_{H_{ul}^\mu}.
 \end{align*}
  \end{proof}
Combining this lemma and the proof of Lemma 7.13, \cite{ABZ2}, we obtain
 \begin{lemm}\label{techpara}
Let $w\in \mathcal{W}_{po}(\varrho),~\varrho\ge 0$. Let $\chi \in C_0^\infty(\xR^d)$ and $\widetilde{\chi} \in C_0^\infty(\xR^d)$ be equal to one on the support of $\chi$. Let $\psi, \theta \in \mathcal{S}(\xR^d)$. For every $m, \sigma  \in \xR$ there exists a constant $C>0$ such that
\begin{equation}\label{est:tech1}
 \sum_{j\geq -1} \Vert w\chi_k \psi(2^{-j}D)((1-\widetilde{\chi}_k)u) \theta(2^{-j}D)v \Vert_{H^m(\xR^d)} \leq C \Vert wu \Vert_{H^\sigma_{ul}(\xR^d)}  \Vert v\Vert_{L^\infty(\xR^d)}.
 \end{equation}
For every $m, \sigma, t \in \xR$ one can find a constant $C>0$ such that
  \begin{equation}\label{est:tech2}
 \sum_{j\geq -1} \Vert w\chi_k \psi(2^{-j}D)((1-\widetilde{\chi}_k)u) \theta(2^{-j}D)v \Vert_{H^m(\xR^d)} \leq C \Vert wu \Vert_{H^\sigma_{ul}(\xR^d)}  \Vert v\Vert_{H^t_{ul}(\xR^d)}
 \end{equation}
and 
  \begin{equation}\label{est:tech3}
 \sum_{j\geq -1} \Vert w\chi_k \psi(2^{-j}D)((1-\widetilde{\chi}_k)u) \theta(2^{-j}D)\widetilde\chi_kv \Vert_{H^m(\xR^d)} \leq C \Vert wu \Vert_{H^\sigma_{ul}(\xR^d)}  \Vert \widetilde \chi_kv\Vert_{H^t(\xR^d)}.
\end{equation}
 \end{lemm}
\begin{rema}
It follows easily from the proof of the above lemma that the same estimates as  in \eqref{est:tech1}, \eqref{est:tech2} and \eqref{est:tech3} hold if on the left-hand sides $2^{-j}$ is replaced by $2^{-j-j_0}$ where $j_0\in \xZ$ is fixed. We shall use this remark to deal with paraproduct estimates.
\end{rema}
It turns out that the symbolic calculus with weights possesses the same features as in the usual setting.
 \begin{theo}\label{calc:symb}
Let $m,m' \in \xR,~ \rho \ge 0$ and $w\in \mathcal{W}_{po}(\varrho),~\varrho\ge 0$.\\
$(i)$  If $a \in \Gamma^m_0(\xR^d),$ then for all $\mu \in \xR$, there exist a constant $C>0$ such that 
\[
\Vert wT_au\Vert _{H_{ul}^{\mu}(\xR^d)}\le CM_0^m(a, \mu)_{\varrho}\Vert wu\Vert _{H_{ul}^{\mu+m}(\xR^d)}.
\]
$(ii)$ If $a \in \Gamma^m_\rho(\xR^d)$, $b\in \Gamma^{m'}_\rho(\xR^d)$ then, for all $\mu \in \xR$,  there exist a constant $C>0$ such that 
\begin{align*}
&\Vert w(T_a T_b - T_{a\sharp b})u\Vert _{H_{ul}^{\mu}(\xR^d)}&\\
&\le C\left(M^m_\rho(a, \mu)_{\varrho}M^{m'}_0(b, \mu)_{\varrho} +M^m_0(a, \mu)_{\varrho}M^{m'}_\rho(b, \mu)_{\varrho}\right)\Vert wu\Vert _{H_{ul}^{\mu+m+m'-\rho}(\xR^d)}
\end{align*}
with 
\[
a\sharp b:=\sum_{|\alpha|<\rho}\frac{(-i)^{|\alpha|}}{\alpha !}\partial_{\xi}^{\alpha}a(x, \xi)\partial_x^{\alpha}b(x, \xi).
\] 
$(iii)$ Let $a \in \Gamma^m_\rho(\xR^d)$  and denote by $T_a^*$ the adjoint operator of $T_a$ and by $a^*$ the complex conjugate of $a$ (in case $a$ is a matrix, $a^*$ is its conjugate transpose). Then for all $\mu \in \xR$  there exists a constant $C>0$ such that
\[
\Vert w(T_a^* -T_b)u\Vert _{H_{ul}^{\mu}(\xR^d)}\le CM_\rho^m(a, \mu)_{\varrho}\Vert wu\Vert _{H_{ul}^{\mu+m-\rho}(\xR^d)}.
\]
with
\[
b=\sum_{|\alpha|<\rho}\frac{(-i)^{|\alpha|}}{\alpha !}\partial_{\xi}^{\alpha}\partial_x^{\alpha}a^*(x, \xi).
\]
\end{theo}
\begin{proof}
 We give the proof for the first assertion only since these three points are proved along the same lines. For simplicity we shall consider symbols in $\dot{\Gamma}^m_\rho(\xR^d)$.\\ {\it Step 1.} Consider first the case where $a$ is a bounded function and write
$$\chi_kw T_a u = \chi_k wT_a (\widetilde{\chi}_k u) + \chi_kwT_a ((1- \widetilde{\chi}_k)u).$$
The classical theory gives
\[
\begin{aligned}
 \Vert \chi_k wT_a (\widetilde{\chi}_k u) \Vert_{H^\mu } &\leq C w(k)\Vert a \Vert_{L^\infty } \Vert  \widetilde{\chi}_k u  \Vert_{H^\mu } \leq C  w(k)\Vert a \Vert_{L^\infty } w(k)^{-1}\Vert  \widetilde{\chi}_kwu \Vert_{H^\mu}\\
&\le C\Vert a \Vert_{L^\infty }\Vert  wu \Vert_{H^\mu_{ul}}.
\end{aligned}
\]
The estimate for the second term follows immediately from (\ref{est:tech1}).\\
{\it Step 2.} Next we consider the case $a(x,\xi) = b(x) h(\xi)$ where $h(\xi) = \vert \xi \vert^m\widetilde{h}\big (\frac{\xi}{\vert \xi \vert}\big)$ with $\widetilde{h} \in C^\infty(\xS^{d-1})$. Then directly from the definition we have $T_a = T_b (\psi h)(D_x) $, where the cut-off $\psi$ is given by \eqref{cutoff:psi}. The desired estimate in $(i)$ follows from Step 1 and Proposition \ref{pseudoh}.\\
{\it Step 3.}  Finally, for the general case we introduce $(\widetilde{h}_\nu)_{\nu \in \xN^*}$ an orthonormal basis of $L^2(\xS^{d-1})$ consisting of eigenfunctions of the (self adjoint) Laplace Beltrami operator $\Delta_\omega = \Delta_{\xS^{d-1}} $ on  $L^2(\xS^{d-1})$ i.e. $\Delta_\omega \widetilde{h}_\nu = \lambda^2_\nu \widetilde{h}_\nu$. Setting $h_\nu = \vert \xi \vert^m \widetilde{h} (\omega )$, 
$\omega = \frac{\xi}{\vert \xi \vert}$ when $ \xi \neq 0,$ we can write
$$ a(x,\xi) = \sum_{\nu \in \xN^*}b_\nu(x) h_\nu(\xi) \quad \text{where} \quad b_\nu(x) = \int_{\xS^{d-1}} a(x,\omega) \overline{\widetilde{h}_\nu(\omega)} d\omega.$$
With $I=I(d, m, \mu)_{\varrho}$ we have
$$\lambda_\nu^Ib_\nu(x) = \int_{\xS^{d-1}} \Delta_\omega^{\frac{I}{2}}a(x,\omega) \overline{\widetilde{h}_\nu(\omega)} d\omega, $$
 which gives
\begin{equation} 
\Vert b_\nu \Vert_{L^\infty(\xR^d)} \leq C \lambda_\nu^{-I} M^m_0(a, \mu)_{\varrho}. 
\end{equation}
By definition of $I$, we can find an integer $r$ such that
\[
\lceil m\rceil+\lceil\mu\rceil+\lceil-(m+\mu)\rceil+\lceil\varrho\rceil+\frac{3d}{2}+1<r<I-d.
\]
 By the Weyl formula we know that $\lambda_\nu \sim c\nu^{\frac{1}{d}}$. In addition, there exists a positive constant $K$ such that for all $\nu \geq 1$
\begin{equation}\label{hnu}
\Vert \widetilde{h}_\nu \Vert_{H^{r}(\xS^{d-1})} \leq K \lambda_\nu^{r}.
\end{equation}
 Now using the  steps above and Proposition~\ref{pseudoh} we obtain ($\psi$ is given by \eqref{cutoff:psi})
 \begin{align*}
  \Vert wT_a u \Vert_{H^\mu_{ul} } &\leq \sum_{\nu \geq 1} \Vert T_{b_\nu}(\psi h_\nu)(D_x) u \Vert_{H^\mu_{ul} }\\
  &\leq C \sum_{\nu \geq 1}\Vert b_\nu \Vert_{L^\infty(\xR^d)}\Vert \widetilde{h}_\nu\Vert_{H^{r}(\xS^{d-1})} \Vert wu \Vert_{H^{\mu + m}_{ul} }\\
&\les  M^m_0(a, \mu)_{\varrho} \Vert wu \Vert_{H^{\mu + m}_{ul} } \sum_{ \nu \geq 1}\nu^{\frac{-I+r}{d}} \\
&\les M^m_0(a, \mu)_{\varrho} \Vert wu \Vert_{H^{\mu + m}_{ul}}.
\end{align*}
\end{proof}
\subsubsection{Paraproducts}
\begin{prop}\label{paraproduct}
Let $w\in \mathcal{W}_{po}(\varrho),~\varrho\ge 0$. Let $s_0, s_1, s_2\in \xR$ satisfying $s_0\le s_2$ and $s_0<s_1+s_2-\frac{d}{2}$. Then there exists $C>0$ such that
\[
\|wT_au\|_{H^{s_0}_{ul}}\le C\min\left\{\|a\|_{H^{s_1}_{ul}}\|wu\|_{H^{s_2}_{ul}}, \|wa\|_{H^{s_1}_{ul}}\|u\|_{H^{s_2}_{ul}}\right\}.
\]
\end{prop}
\begin{proof}
We write
\bq\label{writeTau}
\chi_kwT_au=\chi_kwT_a(1-\widetilde\chi_k)u+\chi_kwT_{\widetilde\chi_ka}\widetilde\chi_ku+\chi_kwT_{(1-\widetilde\chi_k)a}\widetilde\chi_ku.
\eq
By the classical result, we have
\bq\label{wchi}
\Vert \chi_kwT_{\widetilde\chi_ka}\widetilde\chi_ku\Vert _{H^{s_0}}\lesssim w(k)\Vert T_{\widetilde\chi_ka}\widetilde\chi_ku\Vert _{H^{s_0}}
\lesssim  w(k)\|\widetilde\chi_ka\|_{H^{s_1}}\|\widetilde\chi_ku\|_{H^{s_2}}
\lesssim \|a\|_{H^{s_1}_{ul}}\|wu\|_{H^{s_2}_{ul}}.
\eq
On the other hand, applying \eqref{est:tech2} gives
$$
\Vert \chi_kwT_a(1-\widetilde\chi_k)u\Vert _{H^{s_0}}\lesssim\|a\|_{H^{s_1}_{ul}}\|wu\|_{H^{s_2}_{ul}}
$$
and it follows form \eqref{est:tech3} (applied with $w\equiv 1$) that
$$
\begin{aligned}
\Vert  \chi_kwT_{(1-\widetilde\chi_k)a}\widetilde\chi_ku\Vert _{H^{s_0}}&\lesssim w(k)\Vert  \chi_kT_{(1-\widetilde\chi_k)a}\widetilde\chi_ku\Vert _{H^{s_0}}\\
&\lesssim w(k)\|a\|_{H^{s_1}_{ul}}\|\widetilde\chi_ku\|_{H^{s_2}}
\lesssim \|a\|_{H^{s_1}_{ul}}\|wu\|_{H^{s_2}_{ul}}.
\end{aligned}
$$
Consequently, we obtain 
\[
\|wT_au\|_{H^{s_0}_{ul}}\le C\|a\|_{H^{s_1}_{ul}}\|wu\|_{H^{s_2}_{ul}}.
\]
Now if instead of (\ref{writeTau}), we decompose
\[\label{writeTau}
\chi_kwT_au=\chi_kwT_{(1-\widetilde\chi_k)a}u+\chi_kwT_{\widetilde\chi_ka}\widetilde\chi_ku+\chi_kwT_{\widetilde\chi_ka}(1-\widetilde\chi_k)u
\]
then we get
\[
\|wT_au\|_{H^{s_0}_{ul}}\le C\|wa\|_{H^{s_1}_{ul}}\|u\|_{H^{s_2}_{ul}}.
\]
The proof is complete.
\end{proof}
 \begin{prop}\label{paralin}
Let $w\in \mathcal{W}_{po}(\varrho),~\varrho\ge 0$ and two functions $a\in H ^{s_1}_{ul}(\xR^d), u\in H^{s_2}_{ul}(\xR^d)$. \\
$(i)$ If  $s_1+s_2>0$ then 
 \begin{equation}\label{R(a,u)}
  \Vert wR(a,u) \Vert_{H_{ul}^{s_1+s_2 - \frac{d}{2}}(\xR^d)} \leq C\Vert a\Vert_{ H^{s_1}_{ul}(\xR^d)} \Vert wu\Vert_{ H^{s_2}_{ul}(\xR^d)}.
  \end{equation}
$(ii)$ If $s_1+s_2>0,~s_0\le s_1$ and $s_0<s_1+s_2 - \frac{d}{2}$ then there exists a constant $C>0$ such that
 \begin{align}\label{est:a-Ta}
 \Vert w(a-T_a)u \Vert_{H^{s_0}_{ul}(\xR^d)} \leq C \min\left\{\Vert a \Vert_{H^{s_1}_{ul}(\xR^d)}\Vert  wu \Vert_{H^{s_2}_{ul}(\xR^d)}, \Vert wa \Vert_{H^{s_1}_{ul}(\xR^d)}\Vert  u \Vert_{H^{s_2}_{ul}(\xR^d)}\right\}.
 \end{align}
$(iii)$ If $s_1+s_2>0,~s_0\le s_1,~s_2$ and $s_0<s_1+s_2 - \frac{d}{2}$ then there exists a constant $C>0$ such that
 \begin{align}\label{product:w}
 \Vert wau \Vert_{H^{s_0}_{ul}(\xR^d)} \leq C \min\left\{\Vert a \Vert_{H^{s_1}_{ul}(\xR^d)}\Vert  wu \Vert_{H^{s_2}_{ul}(\xR^d)}, \Vert wa \Vert_{H^{s_1}_{ul}(\xR^d)}\Vert  u \Vert_{H^{s_2}_{ul}(\xR^d)}\right\}.
 \end{align}
\end{prop}
\begin{proof}
$(i)$ By definition, we have (for some cut-off function $\varphi$)
$$R(a,u) = \sum_{j\geq -1} \sum_{\vert k-j\vert \leq 1} \varphi(2^{-j}D)a \cdot\varphi(2^{-k}D)u.$$
We write $a= \widetilde{\chi}_k a + (1-\widetilde{\chi}_k)a, u=  \widetilde{\chi}_k u + (1-\widetilde{\chi}_k)u$ so that
$$\chi_kwR(a,u) = \chi_kwR(\widetilde{\chi}_ka,\widetilde{\chi}_ku) + \chi_kwS_k(a,u).$$
The first term is estimated by the same method as (\ref{wchi}) with the use of Theorem 2.9 $(i)$ in \cite{ABZ1}. The remainder $w\chi_kS_k(a,u)$ is estimated by using \eqref{est:tech2} and \eqref{est:tech3}.\\
$(ii)$  and $(iii)$ are direct consequences of $(i)$ and Proposition \ref{paraproduct}.
\end{proof}
\begin{rema}
We remark that with the methods in the proofs above, the commutator estimate in Lemma $7.20$, \cite{ABZ2} still holds for uniformly local Sobolev spaces with the weight $w\in \mathcal{W}_{po}(\varrho),~\varrho\ge 0$.
\end{rema}
\subsection{Transport equations}
For the sake of completeness we present a weighted version of Lemma $7.19$, \cite{ABZ2} on transport equations in uniformly local Sobolev spaces. The result holds for weights in a large class than $\mathcal{W}$ while its proof is a direct consequence of Lemma $7.19$, \cite{ABZ2}.\\
We denote by $\mathcal{W}_0$ the class of all functions $w:\xR^d\to (0, \infty)$ such that $\frac{\nabla w}{w}\in C^\infty_b(\xR^d)$.
\begin{lemm}
Let $I=[0, T],~s_0>1+{\frac{d}{2}}$, $\mu>0$ and $w\in \mathcal{W}_0$. Then there exists $\cF:\xR^+\to \xR^+$ non decreasing such that for $V_j\in L^\infty(I, H^{s_0}(\xR^d))_{ul}\cap L^\infty(I, H^\mu(\xR^d))_{ul}$ $j=\overline{1, d}$, $wf\in L^1(I, H^\mu(\xR^d))_{ul}$, $wu_0\in H^\mu_{ul}(\xR^d)$ and any solution $u$ with $wu\in L^\infty(I, H^{s_0}(\xR^d))_{ul}$ to the problem
\bq\label{eq:transport}
(\partial_t +V\cdot \nabla)u=f,\quad u\arrowvert_{t=0}=u_0
\eq
we have 
\begin{multline}\label{transport}
\Vert wu\Vert_{L^\infty(I, H^{\mu})_{ul}}\le \cF\big(T\Vert V\Vert_{L^\infty(I, H^{s_0})_{ul}}\big)\Big\{\Vert wu_0\Vert_{H^\mu_{ul}}+\Vert wf\Vert_{L^1(I, H^{\mu})_{ul}}\Big. \\ \Big.+\sup_{k\in \xZ^d}\int_0^T\|wu(\sigma)\|_{H^{s_0}_{ul}}\|\widetilde\chi_kV(\sigma)\|_{H^\mu}d\sigma\Big\}.
\end{multline}
\end{lemm}
\begin{proof}
If $u$ is a solution to the transport problem \eqref{eq:transport} then $wu$ is a solution to 
\[
(\partial_t +V\cdot \nabla)(wu)=g,\quad wu\arrowvert_{t=0}=u_0
\]
with $g:=f+uV\cdot\nabla w=f+uwV\cdot r$, $r=\frac{\nabla w}{w}\in C^\infty_b(\xR^d)$. We are in position to apply Lemma $7.19$, \cite{ABZ2} to have
\begin{multline*}
\Vert wu\Vert_{L^\infty(I, H^{\mu})_{ul}}\le \cF\big(T\Vert V\Vert_{L^\infty(I, H^{s_0})_{ul}}\big)\Big\{\Vert wu_0\Vert_{H^\mu_{ul}}+\Vert wg\Vert_{L^1(I, H^{\mu})_{ul}}\Big. \\ \Big.+\sup_{k\in \xZ^d}\int_0^T\|wu(\sigma)\|_{H^{s_0}_{ul}}\|\widetilde\chi_kV(\sigma)\|_{H^\mu}d\sigma\Big\}.
\end{multline*}
On the other hand, for $k\in \xZ^d$ using the product rule \eqref{product:rule} one has with $n$ large enough
\[
\Vert \chi_k wu(\sigma)V(\sigma)\cdot r\Vert_{H^\mu}\le C\Vert \widetilde\chi_k wu(\sigma)V(\sigma)\Vert_{H^\mu}\Vert \chi_k r\Vert_{H^n}
\]
Since $r\in C^\infty_b(\xR^d)$, $\Vert \chi_k r\Vert_{H^n}$ can be bounded by a constant independent of $k$. 
 Finally, applying once again the product rule \eqref{product:rule} (remark that $s_0>1+{\frac{d}{2}}$) gives
\begin{align*}
\Vert \widetilde\chi_k wu(\sigma)V(\sigma)\Vert_{H^\mu}&\le \sum_{|q-k|\le M}\Vert\chi_qwu(\sigma)\Vert_{H^{s_0}}\Vert \widetilde\chi_k V(\sigma)\Vert_{H^\mu}\\
&\le C\Vert wu(\sigma)\Vert_{H^{s_0}_{ul}}\Vert \widetilde\chi_k V(\sigma)\Vert_{H^\mu}.
\end{align*}
Consequently,
\begin{align*}
\Vert wuV\cdot r\Vert_{L^1(I, H^\mu)_{ul}}&=\sup_{k\in \xZ^d}\int_0^T\Vert \chi_k wu(\sigma)V(\sigma)\cdot r\Vert_{H^\mu}d\sigma\\
&\le C\sup_{k\in \xZ^d}\int_0^T\Vert wu(\sigma)\Vert_{H^{s_0}_{ul}}\Vert \widetilde\chi_k V(\sigma)\Vert_{H^\mu}d\sigma,
\end{align*}
from which the estimate \eqref{transport} follows.
\end{proof}

\end{document}